\pdfoutput=1

\documentclass[11pt]{amsart}

\usepackage{amsmath,  amssymb}

\usepackage{amsopn}

\usepackage[T1]{fontenc}
\usepackage[latin1]{inputenc}
\usepackage[matrix,arrow]{xy}
\usepackage{graphicx}
\usepackage{scalerel}
\usepackage{combelow} 

\usepackage{enumitem} 
\setlist[0]{itemsep=5pt}

\usepackage{cite}  
\usepackage[pdftex]{hyperref}

\numberwithin{equation}{section}

\numberwithin{subsection}{section}

\expandafter\let\csname ver@amsthm.sty\endcsname\relax
\let\theoremstyle\relax
\let\newtheoremstyle\relax

\let\qedhere\relax

\usepackage{amsthm}
\usepackage[capitalize,nameinlink]{cleveref}

\makeatletter
\renewcommand{\thm@space@setup}{
  \thm@preskip=.5\baselineskip\@plus.2\baselineskip \@minus.2\baselineskip
  \thm@postskip=\thm@preskip
}
\newtheoremstyle{amsplain}
{\thm@preskip}
{\thm@postskip}
{\itshape}
{\parindent}
{\scshape}
{.}
{ }
{}
\makeatother

\newcommand\transpose{\intercal}

\usepackage{pgffor}
\foreach \x in {C,D,E,N,Q,Z,D,R,T,W}{\expandafter\xdef\csname\x\x\endcsname{\noexpand\mathbb{\x}}}

\foreach \x in {A,B,C,D,F,I,J,K,S,T,V,X,Y,Z}{\expandafter\xdef\csname c\x\endcsname{\noexpand\mathcal{\x}}}

\foreach \x in {A,B,D,F,K,S,T,X}{\expandafter\xdef\csname m\x\endcsname{\noexpand\mathfrak{\x}}}

\newcommand\style{\mathcal }          
\newcommand\oss[1]{{\style O{}\style S}(#1)}
\newcommand\osss{{\style S}}

\newcommand\abajo{\\[.2cm]}
\newcommand\id{\text{id}}

\DeclareMathOperator{\re}{Re} 

\DeclareMathOperator{\spann}{span} \DeclareMathOperator{\tr}{Tr}

\DeclareMathOperator{\dist}{\text{dist}\,}

\theoremstyle{amsplain} 
\newtheorem{theorem}{Theorem}[section]
\newtheorem{lemma}[theorem]{Lemma}
\newtheorem{corollary}[theorem]{Corollary}
\newtheorem{proposition}[theorem]{Proposition}

\theoremstyle{definition}
\newtheorem{definition}[theorem]{Definition}
\newtheorem{remark}[theorem]{Remark}

\begin{document}

\allowdisplaybreaks

\title{The Matricial Range of $E_{21}$}
\author{Martín Argerami}
\address{Department of Mathematics and Statistics \\ University of Regina\\ 3737 Wascana Pwy\\ Regina, SK S4S0A2\\ Canada}
\email{argerami@uregina.ca}
\date{}

\begin{abstract}
The matricial range of the $2\times2$ matrix $E_{21}$ (i.e., the $2\times 2$ unilateral shift) is described very simply: it consists of all matrices with numerical radius at most $1/2$. The known proofs of this simple statement, however, are far from trivial and they depend on subtle results on dilations. We offer here a brief introduction to the matricial range and a recap of those two proofs, following independent work of Arveson and Ando in the early 1970s.
\end{abstract}

\maketitle

\setcounter{tocdepth}{1}

\tableofcontents

\section{Introduction}

In 1969, W. Arveson published a striking paper in Acta Mathematica, called \emph{Subalgebras of C$^*$-algebras} \cite{arveson1969}. In this paper he developed a non-commutative analog of the Choquet theory for function spaces.
His paper is a wonderful mixture of technical prowess and deep thinking about how to rightly generalize certain ideas about function spaces to the non-commutative setting. A key feature of his paper was the use of complete positivity as a non-commutative replacement for the role that positivity has in the commutative case.

A few years later he published an equally remarkable paper \cite{arveson1972}. Besides containing his essential Boundary Theorem, this paper defined the \emph{matricial range} of an operator. As a consequence of an analysis of nilpotent dilations, he was able to explicitly characterize the matricial range of the $2\times 2$ matrix unit $E_{21}$. This is one of the very few non-trivial (that is, non-normal) cases where the matricial range has been determined (the other significant one is the unilateral shift on an infinite-dimensional Hilbert space, which is more or less straightforward).

Almost concurrently, Ando published his results characterizing the numerical range \cite{ando1973}. As a direct byproduct of his results one recovers the characterization of the matricial range of $E_{21}$.

The goal of this article is to describe Arveson and Ando's techniques, together with basic characterizations of the matricial range. The results we offer follow closely the originals, but several of the proofs are new. In the case of Ando, we have also strived to fill in the details from his very condensed arguments.
\section{Preliminaries}

Throughout, $H$ will be a Hilbert space, with inner product ${\langle\cdot,\cdot\rangle}$. We use $B(H)$ to denote the (C$^*$, von Neumann) algebra of bounded operators on $H$; and $K(H)$ for the compact operators. When $\dim H=n$, we canonically identify $H$ with ${\CC^n}$ and ${B(H)}$ with $M_n(\CC)$, the set of $n\times n$ complex matrices. This is done by fixing an orthonormal basis $\{\xi_j\}\subset H$ and considering the rank-one operators
\[
E_{kj}\xi=\langle \xi,\xi_j\rangle\,\xi_k,\ \ \ \ \xi\in H.
\]
These are called \emph{matrix units} and they are characterized up to unitary equivalence (i.e., choice of the orthonormal basis) by the relations
\begin{equation}\label{equation: matrix units}
{E_{kj}}E_{h\ell}=\delta_{jh}\,E_{k\ell},\ \ \ \ E_{kj}^*=E_{jk},\
 \ \ \ \sum_{k=1}^nE_{kk}=I.
\end{equation}
We write $\TT=\{z\in\CC:\ |z|=1\}$, and $\DD=\{z\in\CC:\ |z|<1\}$ for the unit circle and unit disk respectively.
When needed for clarity, we will write ${E_{kj}^n}$ to emphasize that $E_{kj}^n\in M_n(\CC)$. Of particular importance will be, for each $n$, the {unilateral shift}:\[\label{definition: shift}
{S_2}=E_{21}=\begin{bmatrix}0&0\\1&0\end{bmatrix},\ \ {S_n}=\sum_{k=1}^{n-1}E_{k+1,k}=\begin{bmatrix}0&0&0&\cdots&0\\ 1&0&0&\cdots&0\\ 0&1&0&\cdots&0\\
&&&\ddots\\
0&\cdots&0&1&0
\end{bmatrix}.
\]
In the infinite-dimensional case, when $\{\xi_j\}_{j\in\NN}$ is an orthonormal basis of $H$, the associated shift operator $S$ is the linear operator induced by $S:\xi_j\longmapsto\xi_{j+1}$.

An element $x$ of a normed space is said to be \emph{contractive} if $\|x\|\leq1$. A linear map $\phi:\cA\to \cB$ between two C$^*$-algebras is \emph{positive} if it maps positive elements to positive elements; it is \emph{completely positive} if $\phi^{(n)}$ is positive for all $n\in\NN$, with $\phi^{(n)}$  the $n^{\rm th}$ amplification $\phi^{(n)}:M_n(B(H))\to M_n(B(K))$, given by $\phi^{(n)}(A)=[\phi(A_{kj})]_{kj}$. We will mostly consider completely positive maps which are also unital; these are commonly named  \emph{ucp} (unital, completely positive).
The basics of completely positive and completely contractive maps are covered in many texts. We refer the reader to the following canonical three: \cite{Blecher--LeMerdy-book,Paulsen-book,Pisier-book}. We mention one explicit result that we will use:

\begin{proposition}[Choi]\label{proposition: choi characteriztation for the paper}
Let $\phi:M_n(\CC)\to B(H)$ be a linear map. Then $\phi$ is completely positive if and only if

\[
\phi^{(n)}\left(\begin{bmatrix} E_{11}&\cdots&E_{1n} \\ \vdots&\ddots&\vdots\\ E_{n1}&\cdots&E_{nn}\end{bmatrix}\right)\geq0.
\]
\end{proposition}

The \emph{operator system} generated by $T\in B(H)$ is the space $\oss{T}=\spann\{I,T,T^*\}$. More generally, an operator system is a unital selfadjoint subspace $\cS$ of $B(H)$. When one considers ucp maps as morphisms, an operator system $\cS$ can be characterized by its sequence of positive cones $M_n(\cS)_+$. Arveson's Extension Theorem \cite[Theorem 1.2.3]{arveson1969} guarantees that if $\cS\subset B(H)$ is an operator system and  $\phi:\cS\to B(K)$ is completely positive,  there exists a completely positive extension $\tilde\phi:B(H)\to B(K)$. For any fixed operator system $\cS$, the set of ucp maps $\cS\to B(K)$ is BW-compact, where the BW-topology is that given by pointwise weak-operator convergence.
Given $T\in B(H)$, its \emph{numerical range} is the set

\[
\WW_1(T)=\{f(T):\ f \text{ is a state}\}=\{\tr(HT):\ H\geq0,\ \ \tr(H)=1\}.
\]
We note that the equality above is not entirely obvious, since the right-hand-side only accounts for the normal states. But since the normal states are the predual of $B(H)$, any state is a weak$^*$ (that is, pointwise) limit of normal states; so, as $\WW_1(T)$ is closed---by an easy application of Banach--Alaoglu---,  the set of all $f(T)$ where $f$ runs over all the states, is the same as the set of all $f(T)$ where $f$ runs over all the normal states. 

The numerical range is always compact and convex.
The \emph{numerical radius} of $T$ is the number
\[
w(T)=\sup\{|\lambda|:\ \lambda\in \WW_1(T)\}.
\]

\begin{remark}
The numerical range is classically defined as

\[
W(T)=\{\langle T\xi,\xi\rangle:\ \xi\in H\}.
\]
It turns out that $W(T)$ is always dense in $\WW_1(T)$. It is also convex, as proven in the Toeplitz--Hausdorff Theorem. The fact that $\WW_1(T)$ is convex, on the other hand, follows from a straightforward computation.

\end{remark}

As it is common---although not standard---we will refer by ``strong'' convergence of a net, to convergence in the strong operator topology; and by ``weak'' convergence, to convergence in the weak operator topology. 
\section{The Matricial Range}\label{section: matricial range}

The \emph{matricial range} of $T\in B(H)$ is the sequence
\[
{\WW(T)}=\{\WW_n(T): n\in\NN\},
\]
where

\[
{\WW_n(T)}=\{\varphi(T):\ \varphi:\oss{T}\to M_n(\CC)\ \mbox{ is ucp}\}.
\]
In light of Arveson's Extension Theorem, the matricial range of $T$ does not change if we consider $C^*(T)$ or even $B(H)$ as the domain of the ucp maps in the definition of $\WW_n(T)$.
A classic survey on the topic is \cite{farenick1993b}.

One is tempted to include the set
\[
\WW_\infty(T)=\{\phi(T):\ \phi:\oss{T}\to B(\ell^2(\NN))\ \text{ is ucp }\}
\]
(or even higher-dimensional versions in the non-separable case)
in the list $\{\WW_n(T):\ n\in\NN\}$.
But we have the following:
\begin{proposition}\label{proposition: el rango matricial alcanza}
Let $S\in B(H)$, $T\in B(K)$. The following statements are equivalent:

\begin{enumerate}
\item $\WW_n(S)=\WW_n(T)$ for all $n\in\NN$;
\item $\WW_\infty(S)=\WW_\infty(T)$.

\end{enumerate}
\end{proposition}
\begin{proof}
Assume first that $\WW_n(S)=\WW_n(T)$ for all $n\in\NN$. Let $X\in\WW_\infty(S)$. So $X=\phi(S)\in B(\ell^2(\NN))$   for some ucp map $\phi$. Let $\{P_j\}$ be an increasing net of finite-dimensional projections with $P_j\to I$ strongly. Let $k(j)$ be the rank of $P_j$. We can think of $P_jXP_j\in M_{k(j)}(\CC)$. So $P_jXP_j=P_j\phi(S)P_j\in \WW_{k(j)}(S)=\WW_{k(j)}(T)$. Then there exists a ucp map $\psi_j:B(K)\to M_{k(j)}(\CC)$ with $\psi_j(T)=P_jXP_j$. Let $\psi$ be a BW-cluster point of the net $\{\psi_j\}$. Then $\psi(T)=\lim_j\psi_j(T)=\lim_j P_jXP_j=X$, so $X\in \WW_\infty(T)$. We have proven that $\WW_\infty(S)\subset \WW_\infty(T)$, and exchanging roles we get the equality.

Conversely, assume now that $\WW_\infty(S)=\WW_\infty(T)$. Fix $n\in\NN$. Let $X\in\WW_n(S)$. By identifying $M_n(\CC)$ with  the ``upper left corner'' of $B(\ell^2(\NN))$, we may assume $X\in\WW_\infty(S)=\WW_\infty(T)$. Then there exists a ucp map $\psi:B(K)\to B(\ell^2(\NN))$ with $\psi(T)=X$. If $P$ is the projection of rank $n$ such that $X=PXP$, then $P\psi P$ can be seen as a ucp map $B(K)\to M_n(\CC)$. So $X\in \WW_n(T)$. We have proven that $\WW_n(S)\subset\WW_n(T)$, and now reversing roles we get $\WW_n(S)=\WW_n(T)$.
\end{proof}

We will also consider briefly the \emph{spatial matricial range}
\[
\WW^s(T)=\{\WW_n^s(T):\ n\in\NN\},
\] where
\[
\WW_n^s(T)=\{V^*TV:\ V:\CC^n\to H\ \text{ isometry}\,\}.
\]

\bigskip

The following result is due to Bunce--Salinas \cite[Theorem 2.5]{bunce--salinas1976}. The form we use  is taken from \cite[pp. 335--336]{arveson1977}; the proof follows mostly \cite[Lemma II.5.2]{Davidson-book}, but we use a slightly sharper version of Glimm's Lemma than the one used by Davidson.
\begin{lemma}[Bunce--Salinas]\label{theorem: the form of phi(T)}
Let $\phi:B(H)\to M_n(\CC)$ be ucp and such that $\phi(L)=0$ for every compact operator $L$, and let $\cA\subset B(H)$ be a separable C$^*$-algebra. Then there exists a sequence of isometries $V_k:\CC^n\to H$ such that $V_k\to0$ weakly and \[\|\phi(T)-V_k^*TV_k\|\to0,\ \ \ \ T\in\cA.\]
\end{lemma}
\begin{proof}
For a fixed orthonormal basis $\xi_1,\ldots,\xi_n$ of $\CC^n$, consider the map $\Phi:M_n(\cA)\to\CC$ given by
\[
\Phi(A)=\langle \phi^{(n)}(A)\tilde\xi,\tilde\xi\rangle,
\]
where $\tilde\xi=\tfrac1{\sqrt n}\,\begin{bmatrix}\xi_1&\cdots&\xi_n\end{bmatrix}^\transpose \in (\CC^n)^n$. The fact that $\phi$ is ucp makes $\Phi$ a state. By construction, $\Phi(A)=0$ if all entries of $A$ are compact; and the compact operators of $M_n(\cA)$ are precisely the matrices where all entries are  compact. Thus Glimm's Lemma (see \cite[Lemma II.5.1]{Davidson-book}, but here we use the exact form  of \cite[Lemma 1.4.11]{Brown--Ozawa-book}) applies to the C$^*$-algebra $M_n(\cA)$ and the state $\Phi$, and we get an orthonormal sequence vectors $\{\tilde\eta^k\}\subset H^n$, where $\tilde\eta^k=\begin{bmatrix} \eta_1^k&\cdots&\eta_n^k\end{bmatrix}^\transpose \in H^n$ and
\begin{equation}\label{equation: the form of phi(T):1}\Phi(A)=\lim_{k\to\infty}\langle A\tilde\eta^k,\tilde\eta^k\rangle.
\end{equation}
Asymptotically, $\{\sqrt n\,\eta^k_1,\ldots,\sqrt n\,\eta^k_n\}$ is orthonormal; indeed,
using that $\phi$ is unital and so $\phi^{(n)}(I\otimes E_{hj})=I\otimes E_{hj}$,
\begin{align*}
\lim_{k\to\infty}\langle \sqrt n\,\eta^k_j,\sqrt n\,\eta^k_h\rangle
&=n\lim_{k\to\infty}\langle(I\otimes  E_{hj})\tilde\eta^k,\tilde\eta^k\rangle
=n\Phi(I\otimes E_{hj}) \abajo
&=n\,\langle\phi^{(n)}(I\otimes E_{hj})\tilde\xi,\tilde\xi\rangle
=\langle \xi_j,\xi_h\rangle=\delta_{j,h}.
\end{align*}
Define  linear maps $X_k:\CC^n\to H$ by $X_k\xi_j=\sqrt n\,\eta^k_j$. For $\xi,\eta\in\CC^n$, $\varepsilon>0$, and $k$ big enough so that $|\langle\sqrt n\,\eta^k_h,\sqrt n\,\eta^k_j\rangle -\delta_{h,j}|<\varepsilon$,
\begin{align*}
|\langle (X_k^*X_k-I_n)\xi,\eta\rangle |
&=|\langle X_k\xi,X_k\eta\rangle -\langle\xi,\eta\rangle| \abajo
&=\left|\sum_{h,j}\langle\xi,\xi_h\rangle\overline{\langle\eta,\xi_j\rangle}\,
(\langle\sqrt n\,\eta^k_h,\sqrt n\,\eta^k_j\rangle -\delta_{h,j})\right| \abajo
&\leq\varepsilon\,\sum_{h,j}|\langle\xi,\xi_h\rangle\overline{\langle\eta,\xi_j\rangle}|
\leq\varepsilon\,n^2\|\xi\|\,\|\eta\|.
\end{align*}
It follows that $\|X_k^*X_k-I_n\|<\varepsilon\,n^2$, so $\|X_k^*X_k-I_n\|\to0$ as $k\to\infty$. Now, for each $k$,  let $X_k=V_k(X_k^*X_k)^{1/2}$ be the polar decomposition. For $k$ big enough, $X_k^*X_k$ is invertible, so for such $k$,
\[
V_k^*V_k= (X_k^*X_k)^{-1/2}X_k^*X_k(X_k^*X_k)^{-1/2}=I_n.
\]
Hence $V_k:\CC^n\to H$ is an isometry. Also, $\|V_k-X_k\|=\|X_k((X_k^*X_k)^{-1/2}-I_n)\|\to0$.  And $V_k\to0$ weakly since the range of $V_k$ is contained in the range of $X_k$, and $\{\tilde\eta_k\}$ is orthonormal. We have, by \eqref{equation: the form of phi(T):1},
\[
\langle (\phi(T)-X_k^*TX_k)\xi_h,\xi_j\rangle
=n\,\langle \Phi(T\otimes E_{jh})\tilde\xi,\tilde\xi\rangle - n\,\langle (T\otimes E_{jh})\tilde\eta^k,\tilde\eta^k\rangle\xrightarrow[k\to\infty]{}0.
\]As the above convergence is in $M_n(\CC)$, it also holds in norm. Thus
\begin{align*}
\lim_{k\to\infty}\|\phi(T)-V_k^*TV_k\|
&=\lim_{k\to\infty}\|\phi(T)-X_k^*TX_k\|\xrightarrow[k\to\infty]{}0. \qedhere
\end{align*}
\end{proof}

\begin{remark}
Another set considered by Bunce--Salinas \cite{bunce--salinas1976} is the \emph{essential matricial range}. This would be
\[
\WW_n^e(T)=\{\phi(T):\ \phi:B(H)\to M_n(\CC)\ \text{ ucp,}\ \phi|_{K(H)}=0\}.
\]
By \cref{theorem: the form of phi(T)},  $\WW_n^e(T)\subset \overline{\WW_n^s(T)}$.

\end{remark}

In analogy  to the fact that the classical spatial numerical range is dense in the numerical range---minus the fact that $\WW_n^s(T)$ is often not convex---we have the following result. The proof we provide does not follow the original argument.

\begin{proposition}[Bunce--Salinas \cite{bunce--salinas1976}]\label{theorem: Bunce-Salinas spatial matricial range}
Let $T\in B(H)$. Then
\[
\WW_n(T)=\{\sum_k A_k^*XA_k:\ X\in \overline{\WW_n^s(T)},\ A_1,A_2,\ldots\in M_n(\CC), \sum_kA_k^*A_k=I_n\}.
\]
In other words, the C$^*$-convex hull of the closure of the $n^{\rm th}$ spatial matricial range of $T$ equals the $n^{\rm th}$ matricial range of $T$. When $T$ is compact, the closure of $\WW_n^s(T)$ is not needed. If $C^*(T)\cap K(H)=\{0\}$, then $\WW_n(T)=\overline{\WW_n^s(T)}$.

\end{proposition}
\begin{proof}
Fix $n\in\NN$. It is clear that $\WW_n^s(T)\subset \WW_n(T)$. Now consider a ucp map $\phi:B(H)\to M_n(\CC)$, and write $\phi=V^*\pi V$ for a Stinespring dilation, where $V:\CC^n\to K$ is an isometry and $\pi:B(H)\to B(K)$  a representation. We want to show that $\phi(T)$ is of the form $\sum_kA_k^*XA_k$ with $X\in\overline{\WW_n^s(T)}$, and $\sum_kA_k^*A_k^*=I_n$.

Since $\pi$ is bounded, its kernel is a closed two-sided ideal of $B(H)$. Thus there are just two possibilities: either $\pi$ is an isometry, or $\ker \pi=K(H)$.
Assume first that $\pi$ is an isometry. Then  we can identify $\pi(T)$ with $T\otimes I$: indeed, it is not hard to show that there exist a Hilbert space $H_0$ and a unitary $U:K\to H\otimes H_0$ such that $\pi(T)=U^*(T\otimes I)U$. Let $\{F_{st}\}$ denote a set of matrix units for $H_0$, corresponding to the canonical basis $\{f_n\}$ (note that $H_0$ may be finite or infinite-dimensional). Then, with $W:H\otimes \CC f_1\to H$ the isometry $W(\xi\otimes \lambda f_1)=\lambda\xi$,
\[
\langle W^*TW(\xi\otimes\lambda f_1),\eta\otimes\mu f_1\rangle
= \lambda\overline\mu\,\langle T\xi,\eta\rangle =\langle (T\otimes F_{11})(\xi\otimes\lambda f_1),\eta\otimes \mu f_1\rangle.
\]
For each $s$, let $H_s\subset H$ be the subspace $H_s=W(I\otimes F_{1s})UV\CC^n$. We obviously have $\dim H_s\leq n$. Let $R_s:\CC^n\to H$ be an isometry that contains $H_s$ in its range. So $R_s^*R_s=I_n$, and $R_sR_s^*$ is a projection that contains $H_s$ in its range. We have
\begin{align*}
\phi(T)&=V^*U^*(T\otimes I)UV=V^*U^*\left(\sum_s T\otimes F_{ss}\right)UV \abajo
&=V^*U^* \left(\sum_s (I\otimes F_{s1})(T\otimes F_{11})(I\otimes F_{1s})\right)UV \abajo
&=\sum_s V^*U^* (I\otimes F_{1s})^*W^*TW(I\otimes F_{1s})UV \abajo
&=\sum_s V^*U^* (I\otimes F_{1s})^*W^*R_s\,(R_s^*TR_s)\,R_s^*W(I\otimes F_{1s})UV.
\end{align*}
The convergence of the sum $\sum_s T\otimes F_{ss}$ is strong, but our last sum occurs in $M_n(\CC)$, so the convergence is in norm. Also,
\begin{align*}
\sum_s V^*U^* (I\otimes F_{1s})^*W^*R_sR_s^*W(I\otimes F_{1j})UV
&=V^*U^*\left(\sum_s I\otimes F_{ss}\right) UV \abajo
&=V^*U^*UV=I_n.
\end{align*}
As $R_s^*TR_s\in\WW_n^s(T)$ for all $s$, and $R_s^*W(I\otimes F_{1s})UV\in M_n(\CC)$,
we have shown that \[\phi(T)\in \{\sum_k A_k^*XA_k:\ X\in {\WW_n^s(T)},\ A_1,A_2,\ldots\in M_n(\CC),\ \sum_kA_k^*A_k=I_n\}\]
(note the lack of closure of the spatial matricial range).

In the case where $\pi=0$ on  $K(H)$, we have the same property for $\phi$, and we may apply \cref{theorem: the form of phi(T)} to the separable C$^*$-algebra $C^*(T)$; that way, we obtain isometries $V_k:\CC^n\to H$ with $V_k^*TV_k\to\phi(T)$. So $\phi(T)\in \overline{\WW_n^s(T)}$.

When $T$ is compact, this last case does not apply, and so the closure of $\WW_n^s(T)$ is not needed.

When $C^*(T)\cap K(H)=\{0\}$, the quotient map $\rho:B(H)\to B(H)/K(H)$ is isometric on $C^*(T)$. Thus the map $\phi\circ\rho^{-1}:\rho(C^*(T))\to M_n(\CC)$ is ucp and maps $\rho(T)$ to $\phi(T)$. By Arveson's Extension Theorem, there exists $\tilde\phi:B(H)/K(H)\to M_n(\CC)$, ucp, that extends $\phi\circ\rho^{-1}$. Now $\tilde\phi\circ\rho:B(H)\to M_n(\CC)$ is a ucp map   that annihilates $K(H)$ and such that $\tilde\phi(\rho(T))=\phi(T)$. By \cref{theorem: the form of phi(T)}, $\phi(T)\in\overline{\WW_n^s(T)}$.
\end{proof}

The matricial range was initially defined and studied by Arveson \cite{arveson1972}. It is straightforward to check that each $\WW_n(T)$ is compact and C$^*$-convex (the latter in the sense of \ref{theorem: arveson characterization of matricial range:2:2} in \cref{theorem: arveson characterization of matricial range}), and that $\WW_m(X)\subset\WW_m(T)$ for all $X\in\WW_n(T)$. He mentions, after the definition, that ``it is not hard'' to see that the aforementioned properties characterize the matricial range. The proof we know and write below (\cref{theorem: arveson characterization of matricial range})  is not ``very hard'', but it is not trivial either since it depends on \cref{theorem: Bunce-Salinas spatial  matricial range}, that itself depends on Glimm's Lemma.
Besides Arveson's characterization---\eqref{theorem: arveson characterization of matricial range:2} below---we include a slightly more explicit characterization in terms of finite C$^*$-convex combinations.

\begin{theorem}\label{theorem: arveson characterization of matricial range}
Let $\{\cX_n:\ n\in\NN\}$ be a sequence of sets $\cX_n\subset M_n(\CC)$, and $c>0$. Then the following statements are equivalent:

\begin{enumerate}
\item\label{theorem: arveson characterization of matricial range:1} there exists a  Hilbert space $H$ and $T\in B(H)$ such that $\|T\|\leq c$ and $\cX_n=\WW_n(T)$ for each $n\in\NN$;
\item\label{theorem: arveson characterization of matricial range:2} the sequence $\{\cX_n\}$ satisfies the following properties:

\begin{enumerate}
\item for each $n$, the set $\cX_n$ is compact, and contained in the ball of radius $c$;
\item\label{theorem: arveson characterization of matricial range:2:2} for each $n$, if $X_1,X_2,\ldots\subset \cX_n$ and $A_1,A_2,\ldots\in M_n(\CC)$ with $\sum_k A_k^*A_k=I_n$, then $\sum_k A_k^*X_kA_k\in \cX_n$;
\item\label{theorem: arveson characterization of matricial range:2:3} for each $m,n\in\NN$, $\WW_m(\cX_n)\subset \cX_m$.
\end{enumerate}
\item\label{theorem: arveson characterization of matricial range:3} the sequence $\{\cX_n\}$ satisfies the following properties:

\begin{enumerate}
\item for each $n$, the set $\cX_n$ is compact, and contained in the ball of radius $c$;
\item\label{theorem: arveson characterization of matricial range:3:2} for each $n,m$, if $X_1,X_2,\ldots,X_r\subset \cX_n$ and $A_1,A_2,\ldots,A_r\in M_{n\times m}(\CC)$ with $\sum_k A_k^*A_k=I_m$, then $\sum_k A_k^*X_kA_k\in \cX_m$.

\end{enumerate}
\end{enumerate}
\end{theorem}
\begin{proof}
\eqref{theorem: arveson characterization of matricial range:1}$\implies$\eqref{theorem: arveson characterization of matricial range:3}
If $\cX_n=\WW_n(T)$, as  $\|\phi(T)\|\leq\|T\|\leq c$ for any ucp map $\phi$, it follows that  $\cX_n$ is contained in the ball of radius $c$. Pointwise-norm limits of ucp maps are ucp; so $\cX_n$ is the image of the BW-compact set $\{\phi:B(H)\to M_n(\CC),\ \text{ucp}\}$ under the (continuous) evaluation map $\phi\longmapsto\phi(T)$, and thus compact. If we have a sequence $X_1,X_2,\ldots,X_r\subset \WW_n(T)$, there exist ucp maps $\phi_k$ with $X_k=\phi_k(T)$. For a sequence $A_1,A_2,\ldots,A_r\subset M_{m\times n}(\CC)$ with $\sum_kA_k^*A_k=I_m$, the map $\phi:=\sum_k A_k^*\phi_n (\cdot)A_k$ is ucp, and so $\sum_k A_k^*X_kA_k=\phi(T)\in\WW_m(T)=\cX_m$.

\eqref{theorem: arveson characterization of matricial range:3}$\implies$\eqref{theorem: arveson characterization of matricial range:2} Fix $n$,   $X_1,X_2,\ldots\subset \cX_n$ and matrices $A_1,A_2,\ldots\in M_n(\CC)$ with $\sum_k A_k^*A_k=I_n$. If only finitely many $A_k$ are nonzero, then we get directly that $\sum_kA_k^*X_kA_k\in\cX_n$. So assume that infinitely many $A_k$ are nonzero. Choose $\ell_0$ such that $\|I-\sum_{k=1}^\ell A_k^*A_k\|<1$ for all $\ell>\ell_0$. Then, for $\ell>\ell_0$, $R_\ell=\sum_{k=1}^\ell A_k^*A_k$ is invertible. By \eqref{theorem: arveson characterization of matricial range:3:2}, $\sum_{k=1}^\ell (A_kR_\ell^{-1/2})^*X_kA_kR_\ell^{-1/2}\in\cX_n$ for all $\ell>\ell_0$. As $R_\ell\to I_n$, we also have $R_\ell^{-1/2}\to I_n$ and so, since $\cX_n$ is closed, $\sum_{k=1}^\infty A_k^*X_kA_k\in\cX_n$.  To check that $\WW_m(\cX_n)\subset\cX_m$, let $Y\in \WW_m(\cX_n)$; so there exist $X\in\cX_n$ and $\psi:M_n(\CC)\to M_m(\CC)$, ucp, with $Y=\psi(X)$. By the Kraus Decomposition, we may write $\psi=\sum_{k=1}^r A_k^*\cdot A_k$, with $A_1,\ldots,A_r\in M_{m\times n}(\CC)$. Then $Y=\sum_{k=1}^rA_k^*XA_k\in\cX_m$ by \eqref{theorem: arveson characterization of matricial range:3:2}.

\eqref{theorem: arveson characterization of matricial range:2}$\implies$\eqref{theorem: arveson characterization of matricial range:1} For each $n\in\NN$, let $\{Q_{n,k}\}_{k\in \NN }$ be a countable dense subset of $\cX_n$. Let $H=\ell^2(\NN)$, where we index the canonical orthonormal basis as \[\{\xi_{n,k,j}:\ n,k\in \NN ,\ j=1,\ldots,n\}.\] For each $n,k\in \NN $ we denote the canonical basis in $\CC^n$ by $\delta_1,\ldots,\delta_n$, and we define linear isometries $W_{n,k}:\CC^n\to H$ by

\[
W_{n,k}\delta_j=\xi_{n,k,j},\ \ \ \ j=1,\ldots,n.
\]
We have
\begin{align*}
\langle W_{n,k}^*\xi_{m,\ell,h},\delta_j\rangle 
&=\langle \xi_{m,\ell,h},W_{n,k}\delta_j\rangle
=\langle \xi_{m,\ell,h},\xi_{n,k,j}\rangle \abajo
&=\delta_{m,n}\,\delta_{\ell,k}\,\delta_{h,j}
=\delta_{m,n}\,\delta_{\ell,k}\,\langle \delta_h,\delta_j\rangle,
\end{align*}
So \[W_{n,k}W_{n,k}^*\xi_{m,\ell,h}=\delta_{m,n}\,\delta_{\ell,k}\,W_{n,k}\delta_h
=\delta_{m,n}\,\delta_{\ell,k}\,\xi_{n,k,h}.\] It follows that $W_{n,k}W_{n,k}^*$ is the orthogonal projection onto the span of the vectors $\xi_{n,k,1},\ldots,\xi_{n,k,n}$. Also,

\[
W_{n,k}^*W_{n',k'}\delta_j=W_{n,k}^*\xi_{n',k',j}=\delta_{n,n'}\delta_{k,k'}\delta_j,
\]so $W_{n,k}^*W_{n,k}=I_n$ and $W_{n,k}^*W_{n',k'}=0$ if $n\ne n'$ or $k\ne k'$. Now define

\[
T=\sum_{n\in\NN}\sum_{k\in \NN } W_{n,k}Q_{n,k}W_{n,k}^*.
\]
The series is well-defined---via strong convergence---because the projections $\{W_{n,k}W_{n,k}^*\}$ add to the identity of $H$. It is clear that $\|T\|\leq\max\{\|Q_{n,k}\|:\ n,k\}<c$. For any isometry $V:\CC^m\to H$, we have

\begin{align*}
  V^*TV =\sum_{n\in\NN}\sum_{k\in \NN } V^*W_{n,k}Q_{n,k}W_{n,k}^*V.
\end{align*}
\textbf{Claim: }If $X_j\in \cX_{n(j)}$, $A_j\in M_{n(j)\times m}(\CC)$, $j\in \NN$, then $\sum_j A_j^*X_jA_j\in\cX_m$.

As
\[
\sum_{n\in\NN}\sum_{k\in \NN } V^*W_{n,k}W_{n,k}^* V
=V^*\,\left(\sum_{n\in\NN}\sum_{k\in \NN } W_{n,k}W_{n,k}^*\right)\, V
=V^*IV=V^*V=I_m
\]
and $W_{n,k}^*V\in M_{n\times m}(\CC)$ for each $n$, the Claim above implies that $V^*TV\in\cX_m$. It follows that $\overline{\WW_n^s(T)}\subset\cX_n$. Then the C$^*$-convexity \eqref{theorem: arveson characterization of matricial range:2:2}, compactness, and  \cref{theorem: Bunce-Salinas spatial matricial range} give us $\WW_n(T)\subset\cX_n$. From $Q_{n,k}=W_{n,k}^*TW_{n,k}$ we obtain $Q_{n,k}\in \WW_n(T)$ for all $k\in \NN $, since the map $X\longmapsto W_{n,k}^*XW_{n,k}$ is ucp $B(H)\to M_n(\CC)$; the density of $\{Q_{n,k}\}$ then shows that $\cX_n\subset \WW_n(T)$. Thus $\WW_n(T)=\cX_n$ for all $n\in\NN$.

It remains to prove the Claim. We will prove the statement for a finite number of matrices, say $1\leq j\leq r$; if we prove that, then an argument with an $R_\ell$ like in the proof of \eqref{theorem: arveson characterization of matricial range:3}$\implies$\eqref{theorem: arveson characterization of matricial range:2} shows the general result. So fix $j\in\NN$, with $1\leq j\leq r$. Consider the ucp map $\psi_j:M_{n(j)}(\CC)\to M_{n(1)+\cdots+n(r)}(\CC)$ given by
\[
\psi_j(X)=\begin{bmatrix}\psi_{j1}(X) \\ & \psi_{j2}(X)\\ & & \ddots \\ & & & \psi_{jr}(X) \end{bmatrix}
\]
where each $\psi_{jk}:M_{n(j)}(\CC)\to M_{n(k)}(\CC)$ is some ucp map, and $\psi_{jj}$ is the identity, so $\psi_{jj}(X_j)=X_j$. By construction $\psi_j$ is ucp, so  \eqref{theorem: arveson characterization of matricial range:2:3} implies that  $\psi_j(X_j)\in \cX_{n(1)+\cdots+n(r)}$ for each $j=1,\ldots,r$. This implies, by \eqref{theorem: arveson characterization of matricial range:2:2}, that

\begin{align*}
\begin{bmatrix} X_1\\ & X_2\\ & & \ddots \\ & & & X_r\end{bmatrix}
=&\begin{bmatrix}
I_{n(1)}\\ & 0\\ & & \ddots\\ & & & 0
\end{bmatrix}
\psi_1(X_1)
\begin{bmatrix}
I_{n(1)}\\ & 0\\ & & \ddots\\ & & & 0
\end{bmatrix}
\abajo
& \ \ \ \ + \cdots \abajo
& \ \ \ \ + \begin{bmatrix}
0\\ & \ddots\\ & & 0\\ & & & I_{n(r)}
\end{bmatrix}
\psi_r(X_r)
\begin{bmatrix}
0\\ & \ddots\\ & & 0\\ & & & I_{n(r)}
\end{bmatrix} \abajo
&\in \cX_{n(1)+\cdots+n(r)}.
\end{align*}
Now, by \eqref{theorem: arveson characterization of matricial range:2:3},
\begin{align*}
\sum_{j=1}^r A_j^*X_jA_j
= \begin{bmatrix}
A_1\\
A_2\\
\vdots\\
A_r
\end{bmatrix}^*
\begin{bmatrix} X_1\\ & X_2\\ & & \ddots \\ & & & X_r\end{bmatrix}
\begin{bmatrix}
A_1\\
A_2\\
\vdots \\
A_r
\end{bmatrix}\in\cX_m, \end{align*}
since conjugation by $\begin{bmatrix} A_1&\cdots&A_r\end{bmatrix}^*$ is ucp.
\end{proof}
In the conditions of \cref{theorem: arveson characterization of matricial range}, the radius of $\WW_n(T)$ is actually $\|T\|$ for all $n\geq2$. For $n=1$, it is well-known that the radius could be $\tfrac12\,\|T\|$ (as is the case when $T=E_{21}$). Concretely, define
\[
\nu_n(T)=\sup\{\|X\|:\ X\in \WW_n(T)\}.
\]
\begin{proposition}[Smith--Ward \cite{smith--ward1980}]\label{proposition: radius of the matricial range}
Let $T\in B(H)$. Then $\nu_n(T)=\|T\|$ for all $n\geq2$.

\end{proposition}
\begin{proof}
For any ucp map $\phi:B(H)\to M_n(\CC)$, we have $\|\phi(T)\|\leq\|T\|$, so $\nu_n(T)\leq\|T\|$. Conversely, fix $\varepsilon>0$. Choose $\xi\in H$ such that $\|\xi\|=1$ and $\|T\xi\|>\|T\|-\varepsilon$. Let $H_0$ be an $n$-dimensional subspace of $H$ that contains $\xi$ and $T\xi$. We define $\psi:B(H)\to M_n(\CC)$ by $\psi(L)=P_{H_0}LP_{H_0}$ (with the usual identification $P_{H_0}B(H)P_{H_0}\simeq M_n(\CC)$). The map $\psi$ is ucp, and

\[
\nu_n(T)\geq\|\psi(T)\|=\|P_{H_0}TP_{H_0}\|\geq\|T\xi\|>\|T\|-\varepsilon.
\]
As $\varepsilon$ was arbitrary, we get $\nu_n(T)=\|T\|$.
\end{proof}

The family of examples where $\WW(T)$ can be found explicitly is fairly small. The most notable example is $\WW(E_{21})$, as will be established independently in Corollaries \labelcref{corollary: matricial range of the 2x2 shift,corollary: matricial range of E_21}. A small generalization, to quadratic operators, is considered in \cite{TsoWu1999}. As could be expected, though, the case of normal operators is not hard.
\begin{proposition}\label{proposition: matricial range of normal}
Let $T\in B(H)$ be normal, and $n\in\NN$. Then
\[
\WW_n(T)=\overline{\left\{\sum_{j=1}^k\lambda_jH_j: k\in\NN, H_j\geq0,\ \lambda_j\in\sigma(T), \sum_jH_j=I_n \right\} }.
\]
\end{proposition}
\begin{proof}
Since $T$ is normal, we can identify $C^*(T)$ with $C(\sigma(T))$. Given any decomposition $\sum_{j=1}^k\lambda_jH_j$ as above, we can find positive linear functionals $f_j$ (characters, actually) with $f_j(T)=\lambda_j$. The map $\psi:X\longmapsto \sum_{j=1}^k f_j(X)H_j$ is a unital positive linear map $C^*(T)\to M_n(\CC)$. As the domain is abelian, $\psi$ is completely positive \cite[Theorems 3.9 and 3.11]{Paulsen-book}. This shows the inclusion $\supset$ above.
Conversely, let $\phi:B(H)\to M_n(\CC)$ be ucp. Fix $\varepsilon>0$. By the Spectral Theorem, we can find projections $P_1,\ldots,P_k\in B(H)$ with $\sum_jP_j=I$ and $\lambda_1\,\ldots,\lambda_k\in\sigma(T)$ with $\|T-\sum_j\lambda_jP_j\|<\varepsilon$. Then
\[
\|\phi(T)-\sum_j\lambda_j \phi(P_j)\|=\|\phi(T-\sum_j\lambda_jP_j)\|<\varepsilon.
\]
As we can do this for each $\varepsilon>0$, we have shown that $\phi(T)$ is a limit of matrices of the form $\sum_j\lambda_j H_j$ as above.
\end{proof}

We continue with another of Arveson's gems from\cite{arveson1972}. Given $T\in B(H)$, $S\in B(K)$, we say that $S$ is a \emph{compression} of $T$ if there exists a projection $P\in B(H)$ such that $S$ is unitarily equivalent to $PT|_{PH}\in B(PH)$. This definition agrees with the usual use of the word ``compression'' or ``corner'', but it is important to emphasize the restriction aspect: for instance, in $M_2(\mathbb C)$ the projection $E_{11}$ is not a compression of $I_2$ in the above sense; or, for another example, the matrix unit $E_{11}\in M_2(\mathbb C)$ is a compression of $E_{11}\in M_3(\mathbb C)$, but not viceversa. The important thing to note is that the unitary implementing the unitary equivalence maps $K$ onto $PH$.

\begin{theorem}[Arveson]\label{theorem: Arveson metric characterization of inclusion of matricial range}
Let $T\in B(H)$, $S\in B(K)$. The following statements are equivalent:
\begin{enumerate}
\item\label{theorem: Arveson metric characterization of inclusion of matricial range:1} $\WW_n(S)\subset \WW_n(T)$ for all $n\in\NN$;
\item\label{theorem: Arveson metric characterization of inclusion of matricial range:2} for each $n\in\NN$ and $A,B\in M_n(\CC)$, \begin{equation}\label{equation: Arveson metric characterization:1}\|A\otimes I+B\otimes S\|\leq\| A\otimes I+B\otimes T\|;\end{equation}
\item\label{theorem: Arveson metric characterization of inclusion of matricial range:3}for each finite-dimensional projection $P\in B(K)$, there exists a $*$-representation $\pi:C^*(T)\to B(H_\pi)$ such that $PS|_{PK}\in B(PK)$ is unitarily equivalent to a compression of $\pi(T)$;
\item\label{theorem: Arveson metric characterization of inclusion of matricial range:4}there exists a $*$-representation $\pi:C^*(T)\to B(H_\pi)$ such that $S$ is a compression of $\pi(T)$.
\end{enumerate}
Moreover,
\begin{enumerate}[resume]
\item\label{theorem: Arveson metric characterization of inclusion of matricial range:5} When $S$ is normal, the above conditions are equivalent to $\sigma(S)\subset \WW_1(T)$;
\item\label{theorem: Arveson metric characterization of inclusion of matricial range:6} when $T$ is compact and irreducible, the above conditions are equivalent to $S$ being unitarily equivalent to a compression of $T\otimes I$.
\end{enumerate}
\end{theorem}

\begin{proof}
\eqref{theorem: Arveson metric characterization of inclusion of matricial range:1}$\implies$\eqref{theorem: Arveson metric characterization of inclusion of matricial range:2} By repeating the argument in the first paragraph of the proof of \cref{proposition: el rango matricial alcanza}, we can get a ucp map $\phi:\oss{T}\to B(K)$ with $\phi(T)=S$.
Then, for any $A,B\in M_n(\CC)$,

\begin{align*}
\|A\otimes I_K+B\otimes S\|&=\|A\otimes \phi(I_H)+B\otimes \phi(T)\|=\|\phi^{(n)}(A\otimes I_H+B\otimes T)\| \abajo
&\leq\|A\otimes I_H+B\otimes T\|.
\end{align*}
\eqref{theorem: Arveson metric characterization of inclusion of matricial range:2}$\implies$\eqref{theorem: Arveson metric characterization of inclusion of matricial range:3} Define a linear map $\phi:\spann\{I_H,T\}\to\spann\{I_K,S\}$ by $\phi(I_H)=I_k$, $\phi(T)=S$. The condition \eqref{equation: Arveson metric characterization:1} now says that $\phi$ is completely contractive. By \cite[Proposition 3.5]{Paulsen-book}, $\phi$ extends to a ucp map $\phi:\spann\{I_H,T,T^*\}\to B(K)$, and by Arveson's Extension Theorem we may enlarge the domain of $\phi$ to be $B(H)$. Consider a Stinespring Dilation  $\phi=P_K\pi|_K$, where $\pi:B(H)\to B(K_\pi)$ is a representation and $K\subset K_\pi$.  For any projection $P\in B(K)$, since $PK\subset K=P_KK_\pi$,
\[
PS|_{PK}=P\,P_KS|_{PK}=P\,P_K\pi(T)|_{PK}=P\pi(T)|_{PK}.
\]

\eqref{theorem: Arveson metric characterization of inclusion of matricial range:3}$\implies$\eqref{theorem: Arveson metric characterization of inclusion of matricial range:4} Let $\{P_j\}\subset B(K)$ be an increasing net of finite-rank projections that converges strongly to $I$. By hypothesis, for each $j$ there exist a unitary $V_j:P_jK\to Q_jH_j$, a projection $Q_j\in B(H_j)$, and a representation $\pi_j:C^*(T)\to B(H_j)$ such that
\[
P_jS|_{PjK}=V_j^*Q_j\pi_j(T)|_{Q_jH_j}\,V_j.
\]
Fix a state $f$ of $C^*(T)$. Then the maps $\psi_j:C^*(T)\to B(K)$ given by
\[
\psi_j(X)=V_j^*Q_j\pi_j(X)|_{Q_jH_j}\,V_j+f(X)\,(I-P_j)
\]
are ucp. Let $\psi:C^*(T)\to B(K)$ be a BW-cluster point of the net $\{\psi_j\}$. It is clear that $\psi(T)=S$. Now a Stinespring decomposition of $\psi$ gives $S$ as a compression of $T$.

\eqref{theorem: Arveson metric characterization of inclusion of matricial range:4}$\implies$\eqref{theorem: Arveson metric characterization of inclusion of matricial range:1} By hypothesis, there is a ucp map $\psi$ with $S=\psi(T)$. Given any $X\in\WW_n(S)$, there exists a ucp map $\phi$ with $\phi(S)=X$. Then
\[
X=\phi(S)=\phi\circ\psi(T)\in \WW_n(T).
\]

\eqref{theorem: Arveson metric characterization of inclusion of matricial range:5} Assume $S$ is normal. If $\sigma(S)\subset \WW_1(T)$, then  $\WW_1(S)\subset \WW_1(T)$ since $\WW_1(S)$ is the closed convex hull of $\sigma(S)$. This allows us to show that the unital map  $\psi:\oss{T}\to\oss{S}$ with $\psi(T)=S$, $\psi(T^*)=S^*$ is positive: indeed, if $\alpha I+\beta T+\gamma T^*\geq0$ then $\WW_1(\alpha I+\beta T+\gamma T^*)\subset[0,\infty)$ and, as
\begin{align*}
\sigma(\alpha I+\beta S+\gamma S^*) & = \{\alpha+\beta\lambda+\gamma\bar\lambda:\ \lambda\in \sigma(S)\} \abajo
&\subset \{\alpha+\beta\lambda+\gamma\bar\lambda:\ \lambda\in \WW_1(T)\} \abajo
&=\{f(\alpha I+\beta T+\gamma T^*):\ f\ \text{ state}\}
\subset[0,\infty)
\end{align*}
we obtain that $\alpha I+\beta S+\gamma S^*\geq0$ (here we use that $S$ is normal to characterize positivity by its spectrum, and also for the first equality above). So $\psi$ is a unital positive map; as $C^*(S)$ is abelian, $\psi$ is ucp. Now we can use this $\psi$ to check that \eqref{theorem: Arveson metric characterization of inclusion of matricial range:2} or \eqref{theorem: Arveson metric characterization of inclusion of matricial range:4} hold. Conversely, if the equivalent conditions hold, we have $\WW_1(S)
\subset\WW_1(T)$, and so $\sigma(S)\subset\WW_1(S)\subset\WW_1(T)$.

\eqref{theorem: Arveson metric characterization of inclusion of matricial range:6}When $T$ is compact and
\eqref{theorem: Arveson metric characterization of inclusion of matricial range:4} holds, the representation $\pi$ is nonzero on $T$, and so it has to be isometric on $C^*(T)$---as the only possible kernel is $K(H)$. Because $T$ is irreducible, $K(H)\subset C^*(T)$ (this is a more or less straightforward consequence of Kadison's Transitivity Theorem; see \cite[Theorem I.10.4]{Davidson-book}). So $\pi$ is isometric on $B(H)$.  It is well-known   that in this situation $\pi(T)$ is unitarily equivalent to $T\otimes I$. Conversely, if $S$ is unitarily equivalent to a compression of $T\otimes I$, we can recover \eqref{theorem: Arveson metric characterization of inclusion of matricial range:4}.
\end{proof}

It was proven by Hamana \cite{hamana1979b} (see \cite{davidson--kennedy2013} for a bit of history and a proof within Arveson's framework) that every operator system admits a C$^*$-envelope. That is, given an operator system $\osss$, there exists a C$^*$-algebra $\cA$ and a complete isometry $j:\osss\to \cA$ such that for any C$^*$-algebra $B$ and any complete isometry $\psi:\osss\to B$, there exists a C$^*$-epimorphism $\pi:C^*(\psi(\osss))\to \cA$ with $\pi\circ\psi=j$. That is, the following diagram commutes:
\begin{equation}\label{equation: c-star envelope}
\xymatrix{ \osss \ar[rr]^{\psi} \ar[drr]_j & & C^*(\psi(\osss)) \\
& &\cA\ar@{{}{ }{}}[u]_{\phantom{\pi\text{ epimorphism}}}{\ar@{<--}[u]_{\pi\text{ epimorphism}}}}
\end{equation}
It is straightforward to prove that for a given operator system $\osss$ the C$^*$-algebra $\cA$ above is determined up to isomorphism, and so one denotes it by $C^*_e(\osss)$ and names it \emph{the C$^*$-envelope of $\osss$}. By taking the quotient by the kernel of $\pi$ (the \v Silov ideal, in Arveson's terminology), we always have that $C^*_e(T)$ is a quotient of $C^*(T)$. In the particular case where an operator system $\oss{T}\subset B(H)$ has the property that $\pi$ has trivial kernel (so, it is isometric), we say that $T$ is \emph{first order} (this was Arveson's original terminology; in later years he used the word   \emph{reduced}).

Let us now draw some consequences from Arveson's result.

\begin{corollary}\label{corollary: same matricial range}
Let $T\in B(H)$, $S\in B(K)$. The following statements are equivalent:

\begin{enumerate}
\item\label{corollary: same matricial range:1} $\WW(S)=\WW(T)$;
\item\label{corollary: same matricial range:3} for all $n\in\NN$, for any $A,B\in M_n(\CC)$, $\|A\otimes I+B\otimes T\|=\|A\otimes I+B\otimes S\|$;
\item\label{corollary: same matricial range:4} $\oss{S}\simeq\oss{T}$ via a complete isometry $\phi$ with $\phi(S)=T$.

\end{enumerate}
If both $S,T$ are irreducible, first order,  and both $C^*(S)$ and $C^*(T)$ contain a nonzero compact operator, then the above statements are also equivalent to
\begin{enumerate}[resume]
\item\label{corollary: same matricial range:5} $S$ and $T$ are unitarily equivalent.
\end{enumerate}

\end{corollary}
\begin{proof}
The equivalences \eqref{corollary: same matricial range:1}$\iff$\eqref{corollary: same matricial range:3}$\iff$ \eqref{corollary: same matricial range:4} follow directly from \cref{theorem: Arveson metric characterization of inclusion of matricial range}. The implication \eqref{corollary: same matricial range:5}$\implies$\eqref{corollary: same matricial range:4} is trivial.
So assume that $S,T$ are irreducible, first order,  that their C$^*$-algebras contain nonzero compact operators, and that there exists a complete isometry $\phi:C^*(T)\to C^*(S)$ with $\phi(T)=S$.

Considering the diagram \eqref{equation: c-star envelope} for  $\osss=\oss{S}$ and $\psi=\phi^{-1}$, and for $\osss=\oss{T}$ and $\psi=\phi$ respectively, one deduces that $C^*_e(T)\simeq C^*_e(S)$ via an isomorphism $\pi$ with $\pi(T)=S$. As $T$ is irreducible and $C^*(T)$ contains a compact operator, it follows that $K(H)\subset C^*(T)$ (as mentioned above, see \cite[Corollary I.10.4]{Davidson-book}). Similarly, $C^*(S)$ contains all compacts of $ B(K)$. Recall that we are assuming that both $T$ and $S$ are first order, so $C^*(T)=C^*_e(T)$ and $C^*(S)=C^*_e(S)$.

Because $\pi$ is an irreducible representation of $C^*(T)$ and $J=K(H)\subset C^*(T)$ with $\pi|_{J}\ne0$, we have that $\pi|_J$ is irreducible. Indeed, let $\xi\in K$, and consider the subspace $\overline{\pi(J)\xi}\subset K$. Since $J$ is an ideal, $\overline{\pi(J)\xi}$ is invariant for $\pi(C^*(T))=C^*(S)$; as $C^*(S)$ is irreducible, it follows that $\overline{\pi(J)\xi}$ is $K$ or $0$. If it were $0$, we would have $\xi\in[\pi(J)K]^\perp$. But then $\overline{\pi(J)K}\subsetneq K$ and it is invariant for $C^*(S)$---which is irreducible---so $\pi(J)=0$, a contradiction. Thus $\overline{\pi(J)\xi}=K$ for all $\xi$, and so $\pi(J)$ has no reducing subspaces.

Because both the domain and the range of $\pi$ contain their respective compact operators, $\pi$ necessarily maps rank-one projections to rank-one projections. We will show that this implies that $\pi$ is implemented by unitary conjugation. Fix an orthonormal basis $\{\xi_j\}$ of $H$. Choose a unit vector $\eta_1\in \pi(\xi_1\otimes\xi_1)H$, and define
\[
\eta_j=\pi(\xi_j\otimes\xi_1)\eta_1.
\]
One then checks easily that $\{\eta_j\}$ is orthonormal; and it has to be a basis, because a rank-one projection corresponding to a vector orthogonal to $\{\eta_j\}$ would get brought back by $\pi$ to a rank-one projection with range orthogonal to $\{\xi_j\}$, an impossibility. Now define a unitary $U:H\to K$ by
\[
U\xi_j=\eta_j.
\]
Then
\begin{align*}
U^*\pi(\xi_j\otimes\xi_k)U\xi_\ell
&=U^*\pi(\xi_j\otimes\xi_k)\eta_\ell
=U^*\pi(\xi_j\otimes\xi_k)\pi(\xi_\ell\otimes\xi_1)\eta_1 \abajo
&=\delta_{\ell,k}\,U^*\pi(\xi_j\otimes\xi_1)\eta_1 
=\delta_{\ell,k}\,U^*\eta_j=\delta_{\ell,k}\,\xi_j
=(\xi_j\otimes\xi_k)\xi_\ell.
\end{align*}
Thus $U^*\pi U$ is the identity on all rank-one operators; by linearity and continuity, it is the identity on all of $K(H)$. For an arbitrary $X\in C^*(T)$ and $\xi\in H$, let $P$ be the rank-one projection with $P\xi=\xi$. Then
\begin{align*}
U^*\pi(X)U\xi
&=U^*\pi(X)UP\xi=U^*\pi(X)UU^*\pi(P)U\xi \abajo 
&=U^*\pi(XP)U\xi=XP\xi=X\xi.
\end{align*}
So $U^*\pi(X)U=X$, that is $\pi(X)=UXU^*$ for all $X\in C^*(T)$. In particular,
\[
S=\pi(T)=UTU^*. \qedhere
\]
\end{proof}

\begin{remark}
The requirement in \eqref{corollary: same matricial range:4} above that $\phi(S)=T$ cannot be relaxed. For example, with
\[
S=\begin{bmatrix} 1&0\\0&0\end{bmatrix},\ \ \ \ T=\begin{bmatrix}1&0\\0&2\end{bmatrix},
\]
we have $\oss{S}=\oss{T}$, but $\WW_1(S)=[0,1]$ and $\WW_1(T)=[1,2]$. Also, we refer to \cref{corollary: matricial range of the unilateral shift} for examples of operators with the same matricial range but very far from unitarily equivalent.

\end{remark}

\begin{remark}
The conditions in \cref{corollary: same matricial range} do not imply the equality of the spatial matricial ranges $\WW^s(S)$ and $\WW^s(T)$. The equality $\WW^s(S)=\WW^s(T)$ implies $\WW(S)=\WW(T)$ by \cref{theorem: Bunce-Salinas spatial  matricial range}, but the converse is not true. For instance consider $H=K=\ell^2(\mathbb N)$, take $S$ to be the unilateral shift with respect to the canonical basis $\{\xi_k\}$, and let $T$ be the unitary given by $T\xi_k=\gamma_k\,\xi_k$, where $\{\gamma_k\}$ is a dense sequence in $\TT$. By \cref{corollary: matricial range of the unilateral shift} below, we have $\WW(S)=\WW(T)$. Any isometry $V:\CC^2\to H$ is given by $Ve_1=x$, $Ve_2=y$, where $\{x,y\}\subset H$ is orthonormal; we will write $V_{x,y}$ for such an isometry. It is not hard to check that, for any $R\in B(H)$,
\[
V_{x,y}^*RV_{x,y}=\begin{bmatrix} \langle x,Rx\rangle & \langle y,Rx\rangle \\ \langle x,Ry\rangle & \langle y,Ry\rangle \end{bmatrix}.
\]
By choosing the sequence $\{\gamma_k\}$ with $\gamma_1=\gamma_2=1$ and taking $x=\xi_1$, $y=\xi_2$, we get
\[
I_2=\begin{bmatrix} 1&0\\0&1\end{bmatrix}=V_{x,y}^*TV_{x,y}\in \WW_2^s(T).
\]
But, while $I_2\in\overline{\WW_2^s(S)}$, we have  $I_2\not\in\WW_2^s(S)$. Indeed, if we had $I_2=V^*_{x,y}SV_{x,y}$ for orthonormal $x,y\in H$, then $V_{x,y}V_{x,y}^*=P$, the orthogonal projection onto the span of $\{x,y\}$. So
\[
P=V_{x,y}V_{x,y}^*=V_{x,y}I_2V_{x,y}^*=V_{x,y}V_{x,y}^*SV_{x,y}V_{x,y}^*=PSP.
\]
This equality cannot hold, because we would have

\[
PS^*SP=P=PS^*PSP,
\]
which implies that $PS^*(I-P)SP=0$, and so $(I-P)SP=0$, from where $SP=PSP=P$. This would make $x$ and $y$ eigenvectors for $S$, a contradiction.
It is easy to see, on the other hand, that $I_2\in\overline{\WW_2^s(S)}$, so it is not clear at first sight whether $\overline{\WW_n^s(T)}\ne\overline{\WW_n^s(S)}$ or not.

\end{remark}

Let us now specialize the above result to the case of matrices. We mention \cite{farenick2011b} for a different and detailed proof of the equivalence \eqref{corollary: same matricial range for matrices:3} $\iff$\eqref{corollary: same matricial range for matrices:5} in  \cref{corollary: same matricial range for matrices} below.

\begin{corollary}\label{corollary: same matricial range for matrices}
Let $n\in\NN$ and $S,T\in M_n(\CC)$. The following statements are equivalent:

\begin{enumerate}
\item\label{corollary: same matricial range for matrices:1} $\WW(S)=\WW(T)$;
\item\label{corollary: same matricial range for matrices:2} $\WW_n(S)=\WW_n(T)$;
\item\label{corollary: same matricial range for matrices:3} for all $A,B\in M_n(\CC)$, $\|A\otimes I+B\otimes T\|=\|A\otimes I+B\otimes S\|$;
\item\label{corollary: same matricial range for matrices:4} $\oss{S}\simeq\oss{T}$ via a complete isometry $\phi$ with $\phi(T)=S$.
\end{enumerate}
If  both $S,T$ are irreducible, the above statements are also equivalent to
\begin{enumerate}[resume]
\item\label{corollary: same matricial range for matrices:5} $S$ and $T$ are unitarily equivalent.
\end{enumerate}
\end{corollary}

\begin{proof}
The equivalences \eqref{corollary: same matricial range for matrices:1} $\iff$\eqref{corollary: same matricial range for matrices:3}$\iff$\eqref{corollary: same matricial range for matrices:4}$\iff$\eqref{corollary: same matricial range for matrices:5} follow directly from \cref{corollary: same matricial range} (note that $M_n(\CC)$ is simple, so the irreducibility of $S$ and $T$ imply that they are first order). The implication \eqref{corollary: same matricial range for matrices:1}$\implies$\eqref{corollary: same matricial range for matrices:2} is trivial, so all that remains is to prove \eqref{corollary: same matricial range for matrices:2}$\implies$\eqref{corollary: same matricial range for matrices:1}. Assume that $\WW_n(S)=\WW_n(T)$, and let $X\in\mathbb W_m(S)$. So there exists a ucp map  $\varphi:M_n(\mathbb C)\to M_m(\mathbb C)$ with $\varphi(S)=X$.
As $S\in\mathbb W_n(S)=\mathbb W_n(T)$, there exists a ucp map $\psi:M_n(\mathbb C)\to M_n(\mathbb C)$ with $\psi(T)=S$. Then
\[
X=\phi(S)=\phi(\psi(T))\in\WW_m(T).
\]
It follows that $\mathbb W_m(S)\subset\mathbb W_m(T)$, and by reversing the roles of $S$ and $T$ we get equality.
\end{proof}

\begin{remark}
The reason one requires irreducibility for unitary equivalence in \cref{corollary: same matricial range for matrices} is multiplicity: for an easy example, we can take $S=E_{11}$, $T=E_{11}+E_{22}$ in $M_3(\CC)$ and then $\WW(S)=\WW(T)$ but they are obviously not unitarily equivalent.
\end{remark}

We will later use some sophisticated ideas---mainly by Arveson and by Ando both building on ideas related to dilations---to calculate the matricial range of the $2\times2$ unilateral shift (\cref{corollary: matricial range of the 2x2 shift,corollary: matricial range of E_21}). The matricial range of the $n\times n$ unilateral shift is unknown for $n\geq3$, but the infinite-dimensional unilateral shift (and, a posteriori, proper isometries) can be tackled with a rather direct approach.  We are grateful to D. Farenick for a simplification of our original argument.
\begin{proposition}\label{proposition: matricial range of the unilateral shift}
Let $T\in B(K)$ with $\|T\|\leq1$, and $S\in B(H)$ the unilateral shift. Then there exists a ucp map $\psi:B(H)\to B(K)$ with $\psi(S)=T$.
\end{proposition}
\begin{proof}
Since $T$ is a contraction, we can construct a unitary
\[
U_0=\begin{bmatrix} T& (I-TT^*)^{1/2} \\ (I-T^*T)^{1/2} & -T^*\end{bmatrix}\in M_2(B(K)).
\]
If $U$ is a {universal unitary} (that is, $\sigma(U)=\TT$), then $C^*(U)=C(\TT)$. As $\sigma(U_0)$ is a compact  subset of $\TT$,  there is a $*$-epimorphism (onto by Tietze's Extension Theorem) acting by restriction:

\[
\pi:C^*(U)=C(\TT)\to C(\sigma(U_0))=C^*(U_0),
\]
with $\pi(U)=U_0$.
Let $\rho:B(H)\to B(H)/K(H)$ be the quotient map. As $S$ becomes a universal unitary in the Calkin algebra, $C^*(\rho(S))\simeq C(\TT)\simeq C^*(U)$. Let $\gamma:C^*(\rho(S))\to C^*(U)$ be a $*$-isomorphism with $\gamma(\rho(S))=U$. Then $\pi\circ\gamma:C^*(\rho(S))\to C^*(U_0)$ is a $*$-epimorphism with $\pi\circ\gamma(\rho(S))=U_0$; in particular, ucp.
Let $\phi:M_2(B(K))\to B(K)$ be the compression to the $1,1$ entry. Then $\psi=\phi\circ\pi\circ\rho:B(H)\to B(K)$ is a ucp map with
\[
\phi\circ\tilde\pi\circ\rho (S)=T. \qedhere
\]
\end{proof}

\begin{corollary}\label{corollary: matricial range of the unilateral shift}
Let $S\in B(H)$ be a proper isometry, or a unitary with full spectrum $\TT$, and let $T\in B(K)$. Then the following statements are equivalent:

\begin{enumerate}
\item\label{proposition: matricial range of the unilateral shift:0} there exists a ucp map $\phi:B(H)\to B(K)$ such that $\phi(S)=T$;
\item\label{proposition: matricial range of the unilateral shift:1} $\|T\|\leq1$.
\end{enumerate}
In other words, $\WW_\infty(S)=\{T:\ \|T\|\leq1\}$, and so for all $n\in\NN$, $\WW_n(S)=\{A\in M_n(\CC):\ \|A\|\leq1\}$.
\end{corollary}
\begin{proof}
\eqref{proposition: matricial range of the unilateral shift:0}$\implies$\eqref{proposition: matricial range of the unilateral shift:1} Since $\|S\|=1$ and $\phi$ is ucp, we get that $\|T\|=\|\phi(S)\|\leq1$.

\eqref{proposition: matricial range of the unilateral shift:1}$\implies$\eqref{proposition: matricial range of the unilateral shift:0} Let $S_0$ be the unilateral shift. By \cref{proposition: matricial range of the unilateral shift} there exists a ucp map $\psi$ with $\psi(S_0)=T$. If $S$ is a proper isometry, by the Wold Decomposition there exist  unitaries $U$ and $W$ such that $S=W^*(U\oplus\bigoplus_j S_0)W$. Then

\[
S\longmapsto WSW^*=U\oplus\bigoplus_j S_0\longmapsto S_0\xrightarrow{\psi} T
\]
is a ucp map.  If $S$ is a unitary with full spectrum, we can proceed as in the proof of \cref{proposition: matricial range of the unilateral shift} to get a ucp map $\phi$ with $\phi(S)=T$.
\end{proof}

\section{Unitary Dilations, Numerical Radius, and Matricial Range}

Advances in operator theory---concretely, about dilations---gave between in late 1960s and the early 1970s several striking characterizations of the numerical radius and, as a byproduct, a characterization of the matricial range of the $2\times 2$ unilateral shift. We will visit, in \cref{section: Ando,section: nilpotent dilations} respectively,  Ando's and Arveson's techniques built upon these theories.

Given a group $G$, a function $T:G\to B(H)$ is said to be \emph{positive-definite} if
\begin{equation}\label{equation: definition positive-definite function on a group}
\sum_{s\in G}\sum_{t\in G}\langle T(t^{-1}s)\xi(s),\xi(t)\rangle\geq0
\end{equation}
for all functions $\xi:G\to H$ of finite support. It is not hard to see that \eqref{equation: definition positive-definite function on a group} implies that $T(s^{-1})=T(s)^*$ for all $s\in G$.

A particular case of a positive-definite function is given by a {unitary representation}. That is, a function $U:G\to B(H)$ such that $U(e)=I$, $U(s)$ is a unitary for all $s\in G$, and $U(st)=U(s)U(t)$ for all $s,t\in G$.
It turns out that one can do a kind of GNS representation for a positive-definite function, and so all positive-definite functions arise from unitary representations. We will only need the particular case where $G=\ZZ$.

\begin{theorem}[Sz.Nagy--Foia\cb{s}]\label{theorem: positive-definite functions are compressions of unitary representations}

Let $H$ be a Hilbert space, and $T:\ZZ\to B(H)$ with $T(0)=I$. The following statements are equivalent:
\begin{enumerate}
\item\label{corollary: positive-definite function on Z:1}There exists a Hilbert space $K\supset H$ and a unitary $U\in B(K)$ such that $T(n)=P_HU^n|_H$ for all $N\in\ZZ$.
\item\label{corollary: positive-definite function on Z:2}$T$ is positive-definite.
\end{enumerate}
\end{theorem}
\begin{proof}
This is \cite[Theorem 7.1]{SzNagy--Foias-book}.
\end{proof}

\begin{remark}\label{remark: E_21 as a corner of the bilateral shift}
In the case where $H=\CC^2$ and $T=T(1)=E_{21}$, $T(n)=0$ for all $n\geq2$, the unitary $U$ can be obtained explicitly as the bilateral shift. Concretely, we take $K=\ell^2(\ZZ)$, and define $U$ on the canonical basis by $Ue_k=e_{k+1}$, extended by linearity and continuity (since $U$ is isometric).  If we identify $H=\CC^2$ with $\spann\{e_0,e_1\}$, then $P_HU^n|_H=E_{12}^n$ for all $n\in\NN$ (which simply means that $P_HU|_H=E_{12}$, $P_HU^2|_H=0$).
\end{remark}

Most considerations of the numerical radius will use the following elementary characterization:
\begin{proposition}\label{proposition: characterization of the numerical range}
Let $T\in B(H)$ with $\|T\|\leq1$. The following statements are equivalent:
\begin{enumerate}
\item\label{proposition: characterization of the numerical range:1} $w(T)\leq1$;
\item\label{proposition: characterization of the numerical range:2} for all $\lambda\in\TT $, $I+\re \lambda
T\geq0$;
\item\label{proposition: characterization of the numerical range:3}for all $\lambda\in\TT$, $\re\lambda T\leq I$;
\item\label{proposition: characterization of the numerical range:4}for all $z\in\DD$, $\re z T\leq I$
\end{enumerate}
\end{proposition}
\begin{proof}
\eqref{proposition: characterization of the numerical range:1}$\implies$\eqref{proposition: characterization of the numerical range:2} If $w(T)\leq1$, then for any $\lambda\in\TT$ and $\xi\in H$ with $\|\xi\|=1$,
\begin{align*}
\langle -\re \lambda T\xi,\xi\rangle &=\re\langle (-\lambda)T\xi,\xi\rangle \abajo
&\leq|\langle(-\lambda)T\xi,\xi\rangle|=|\langle T\xi,\xi\rangle|\leq1=\langle\xi,\xi\rangle.
\end{align*}
Thus $\langle (I+\re\lambda T)\xi,\xi\rangle\geq0$.

\eqref{proposition: characterization of the numerical range:2}$\implies$\eqref{proposition: characterization of the numerical range:3} This is a direct consequence of the  fact that $\TT=-\TT$.

\eqref{proposition: characterization of the numerical range:3}$\implies$\eqref{proposition: characterization of the numerical range:4} If $z\in\DD$, then $z=r\lambda$ with $0\leq r<1$ and $\lambda\in\TT$. So
\[
\re zT=r\,\re\lambda T\leq rI\leq I.
\]

 \eqref{proposition: characterization of the numerical range:4}$\implies$\eqref{proposition: characterization of the numerical range:1} Given $\xi\in H$ with $\|\xi\|=1$, let $\lambda\in\TT$ such that $|\langle T\xi,\xi\rangle|=\lambda\,\langle T\xi,\xi\rangle$. Then
\begin{align*}
|\langle T\xi,\xi\rangle|
&=\re|\langle T\xi,\xi\rangle|=\re \lambda \langle T\xi,\xi\rangle \abajo
&=\langle \re\lambda T\xi,\xi\rangle
=\lim_{r\to1}\langle \re r\lambda T\xi,\xi\rangle
\leq \langle\xi,\xi\rangle =1.
\end{align*}
Thus $w(T)\leq1$.
\end{proof}
\bigskip

We state and prove a version of a characterization of unitary dilations, due to Sz.-Nagy and Foia\cb{s} \cite[Theorem 11.1]{SzNagy--Foias-book}. In their terminology, we only consider 2-dilations.

\begin{theorem}[Sz.Nagy-Foia\cb{s}]\label{theorem: sz.nagy-foias unitary 2-dilation}
Let $T\in B(H)$. The following statements are equivalent:
\begin{enumerate}
\item\label{theorem: sz.nagy-foias unitary 2-dilation:1}there exists a Hilbert space $K\supset H$, and a unitary $U\in B(K)$ such that \begin{equation}\label{equation: unitary 2-dilation}T^n=2\,P_HU^n|_H \ \ \ \text{ for all }n\in\NN;\end{equation}
\item\label{theorem: sz.nagy-foias unitary 2-dilation:2}$w(T)\leq1$.
\end{enumerate}
\end{theorem}
\begin{proof}
\eqref{theorem: sz.nagy-foias unitary 2-dilation:1}$\implies$\eqref{theorem: sz.nagy-foias unitary 2-dilation:2}
Fix $z\in\DD$. Since $|z|<1$, the series below converges and we can manipulate it as follows:
\[
I+2\sum_{k=1}^\infty z^kU^k=-I+2\sum_{k=0}^\infty z^kU^k=-I+2(I-zU)^{-1}=(I+zU)(I-zU)^{-1}.
\]
Then
\begin{equation}\label{equation: sz.nagy-foias unitary 2-dilation:0}
\begin{aligned}
(I_H-zT)^{-1}&=I_H+\sum_{k=1}^\infty z^kT^k=P_H\left.\left( I+2\sum_{k=1}^\infty z^kU^k\right)\right|_H \abajo
&=\left.\vphantom{\int}P_H(I+zU)(I-zU)^{-1}\right|_H.
\end{aligned}
\end{equation}
For any $\xi\in K$,
\[
\re\langle (I+zU)\xi,(I-zU)\xi\rangle=(1-|z|^2)\|\xi\|^2\geq0.
\]
In particular, with $\xi=(I-zU)^{-1}\eta$, we obtain
\begin{equation}\label{equation: sz.nagy-foias unitary 2-dilation:1}
\re\langle (I+zU)(I-zU)^{-1}\eta,\eta\rangle\geq0,\ \ \ \eta\in K,\ z\in\DD.
\end{equation}
When $\xi\in H$, we get from \eqref{equation: sz.nagy-foias unitary 2-dilation:0} and \eqref{equation: sz.nagy-foias unitary 2-dilation:1} that
\begin{align*}
\langle\re(I_H-zT)^{-1}\xi,\xi\rangle
&=\re \langle (I_H-zT)^{-1}\xi,\xi\rangle \abajo
&=\re \langle P_H(I+zU)(I-zU)^{-1}\xi,\xi\rangle \abajo
&=\re \langle (I+zU)(I-zU)^{-1}\xi,\xi\rangle\geq0.
\end{align*}
Replacing $\xi$ with $(I_H-zT)\xi$, the above becomes
\[
\langle \xi,\re(I_H-zT)\xi\rangle\geq0,
\]
and so $\re(I_H-zT)\geq0$. Now \cref{proposition: characterization of the numerical range} gives $w(T)\leq1$.

\eqref{theorem: sz.nagy-foias unitary 2-dilation:2}$\implies$\eqref{theorem: sz.nagy-foias unitary 2-dilation:1} Since $w(T)\leq1$, we have $\sigma(T)\subset \WW_1(T)\subset\overline\DD$. Then  $I-zT=z\,(\frac1z\,I-T)$ is invertible for all $z\in\DD$ (when $z=0$, $I-zT=I$). For any $z\in\DD$,
\[
\|z^nT^n\|^{1/n}\to\text{spr}\,(zT)=|z|\,\text{spr}\,(T)\leq|z|<1.
\]
In particular, for $\varepsilon<1-|z|$, there exists $n_0$ such that $\|z^nT^n\|\leq(|z|+\varepsilon)^n<1$ for all $n\geq n_0$. Thus the series $\sum_{k=0}^\infty z^nT^n$ is norm convergent and equal to $(I-zT)^{-1}$.

Let $\xi\in H$, and put $\eta=(I-zT)^{-1}\xi$. By \cref{proposition: characterization of the numerical range}, we have $\re(I-zT)\geq0$ for all $z\in\DD$; thus
\begin{align*}
0&\leq\re\langle \eta,(I-zT)\eta\rangle= \re \langle  (I-zT)^{-1}\xi,(I-zT)(I-zT)^{-1}\xi\rangle \abajo
&=\re\langle (I-zT)^{-1}\xi,\xi\rangle .
\end{align*}
Define, for $0\leq r<1$ and $0\leq t\leq2\pi$,
\[
Q(r,t)=I+\tfrac12\,\sum_{k=1}^\infty r^k(e^{ikt}T^k+e^{-ikt}T^{*k}).
\]
This converges in norm  by the argument with the spectral radius we just used above.  Also, with $z=r\,e^{it}$,
\[
\langle Q(r,t)\xi,\xi\rangle = \|\xi\|^2+\tfrac12\,\re\,\langle \sum_{k=0}^\infty
z^kT^k\xi,\xi\rangle = \|\xi\|^2+\tfrac12\,\re\langle (I-zT)^{-1}\xi,\xi\rangle \geq0.
\]

Given a sequence $\{\xi_n\}_{n\in\ZZ}$ with finite support, let us form
\[
\xi(t)=\sum_{n\in\ZZ} e^{-int}\xi_n.
\]
Then (recall that the sum over $n$ has finitely many nonzero terms, so the exchange with the integral is not an issue; for the sum over $k$, the convergence is uniform and the exchange is again possible)
\begin{align*}
0&\leq \frac1{2\pi}\int_0^{2\pi}\langle Q(r,t)\xi(t),\xi(t)\rangle\,dt \abajo
&=\sum_{m,n\in\ZZ}\frac1{2\pi}\int_0^{2\pi}e^{-i(n-m)t}\langle \xi_n+\tfrac12\,\sum_{k=1}^\infty r^ke^{ikt}T^k\xi_n+r^ke^{-ikt}T^{*k}\xi_n,\xi_m\rangle\,dt \abajo
&=\sum_{n\in\ZZ}\langle\xi_n,\xi_n\rangle+\tfrac12\sum_{n>m}r^{n-m}\langle T^{n-m}\xi_n,\xi_m\rangle+\tfrac12\,\sum_{n<m}r^{m-n}\langle T^{*(m-n)}\xi_n,\xi_m\rangle
\end{align*}
(for the last equality, note that the integrals will be nonzero only when $-n+m+k=0$ and $-n+m-k=0$; this fixes $k$, and we also have the restriction that $k\geq1$, which makes the case $n=m$ vanish). Now define  a function $T:\ZZ\to B(H)$ by
\[
T(0)=I,\ \ \ T(n)=\tfrac12\,T^n,\ \ \ T(-n)=\tfrac12\,T^{*n},\ \ \ n\geq1.
\]
We can rewrite the inequality above as
\begin{align*}
0\leq&\sum_{n\in\ZZ}\langle T(n-n)\xi_n,\xi_n\rangle+\sum_{n>m}r^{n-m}\langle T(n-m)\xi_n,\xi_m\rangle \abajo
&+\sum_{n<m}r^{m-n}\langle T(n-m)\xi_n,\xi_m\rangle.
\end{align*}
Noting that all sums are finite by hypothesis, we may take $r\nearrow1$, and then
\[
0\leq \sum_{n,m\in\ZZ}\langle T(n-m)\xi_n,\xi_m\rangle,
\]
which shows that the function $n\longmapsto T(n)$ is positive-definite. By \cref{theorem: positive-definite functions are compressions of unitary  representations}, there exists a Hilbert space $K\supset H$ and a unitary $U\in B(K)$ such that $T(n)=P_HU^n|_H$ for all $n\in\ZZ$. In particular,
\[
T^n=2T(n)=2\,P_HU^n|_H,\ \ \ n\in\NN. \qedhere
\]
\end{proof}

\section{Ando's Characterizations of the Numerical Radius}\label{section: Ando}

The following surprising characterization is \cite[Lemma 1]{ando1973}. We do not think that ``lemma'' is a fair word to describe it. The proof follows closely that of Ando, but we have tried to make it a bit clearer. 
\begin{theorem}\label{lemma: ando lemma 1}
Let $T\in B(H)$ with $w(T)\leq1$. Then there exists $X\in B(H)^+$, contractive (i.e., $0\leq X\leq I$) such that for all $\xi \in H$
\begin{equation}\label{equation: ando lemma 1:0}
\langle X\xi,\xi\rangle = \inf\left\{ \left\langle \begin{bmatrix} I & \tfrac12\,T^*  \\ \tfrac12\,T  & X\end{bmatrix}
\begin{bmatrix} \xi\\ \eta\end{bmatrix},\begin{bmatrix} \xi\\ \eta\end{bmatrix}\right\rangle:\ \
\eta\in H\right\}.
\end{equation}
Such $X$ satisfies
\begin{equation}\label{equation: ando lemma 1:00}
X=\max\left\{Y\in B(H):\ 0\leq Y\leq I,\ \begin{bmatrix}I-Y&\tfrac12\,T^* \\ \tfrac12\,T &Y\end{bmatrix}\geq0\right\}.
\end{equation}
\end{theorem}
\begin{proof}
We assume that $w(T)\leq1$. By  \cref{theorem: sz.nagy-foias unitary 2-dilation} there exist a Hilbert space $K\supset H$ and a unitary $U$ in $B(K)$ such that $T^k=2P_HU^k|_H$, $k\in\NN$. We will define a sequence $\{X_n\}\subset B(H)$ in the following way: we start with $X_0=I$, and put $X_n=P_H(I-Q_n)P_H$, where $Q_n$ is the projection onto $\overline\spann\bigcup_{k=1}^nU^{*k}H$. Since by construction these subspaces are increasing on $n$, we have
\[
I=X_0\geq X_1\geq X_2\geq\cdots \geq0.
\]
So the sequence converges strongly to a positive operator $X\in B(H)$---note that we could have defined $X$ directly, but it provides no obvious benefit to the proof.
For $\xi\in H$, we have
\[
\langle U^k\xi,\xi\rangle=\langle U^k\xi,P_H\xi\rangle=\langle P_HU^k|_H\xi,\xi\rangle
=\tfrac12\,\langle T^k\xi,\xi\rangle,\ \ \ \  k\in\NN.
\]Define $A\in M_{n+1}(B(H))$ by \[ A_{kj}=\begin{cases}I,&\ j=k\\ \tfrac12\,T^{*(j-k)},&\ k<j\\ \tfrac12\,T^{(k-j)},&\ k>j\end{cases}
\]
We can write, for $\xi\in H$, and using $\xi_0=\xi$,
\begin{align*}
\langle X_n\xi,\xi\rangle
&=\langle (I-Q_n)\xi,\xi\rangle
=\|(I-Q_n)\xi\|^2
=\|\xi-Q_n\xi\|^2
=\dist(\xi,Q_nH)^2 \abajo
&=\inf\left\{\left\|\xi+\sum_{k=1}^nU^{*k}\xi_k\right\|^2:\ \xi_1,\ldots,\xi_n\in H\right\} \abajo
&= \inf\left\{\left\|\sum_{k=0}^nU^{*k}\xi_k\right\|^2:\ \xi_1,\ldots,\xi_n\in H\right\} \abajo
&= \inf\left\{\sum_{k,j=0}^n\langle U^{*(k-j)}\xi_k,\xi_j\rangle:\ \xi_1,\ldots,\xi_n\in H\right\} \abajo
&= \inf_{\xi_j\in H}\left\{\sum_{k=0}^n\langle\xi_k,\xi_k\rangle
+\sum_{j<k}\langle U^{*(k-j)}\xi_k,\xi_j\rangle
+\sum_{k<j}\langle U^{(j-k)}\xi_k,\xi_j\rangle\right\} \abajo
&= \inf_{\xi_j\in H}\left\{\sum_{k=0}^n\langle\xi_k,\xi_k\rangle
+\sum_{j<k}\langle \tfrac12\,T^{*(k-j)}\xi_k,\xi_j\rangle
+\sum_{k<j}\langle \tfrac12\,T^{(j-k)}\xi_k,\xi_j\rangle \right\} \abajo
&= \inf\left\{\sum_{k,j=0}^n\langle A_{jk}\xi_k,\xi_j\rangle
:\ \xi_1,\ldots,\xi_n\in H\right\} \abajo
&=\inf_{\xi_j\in H}\left\{  \scaleleftright[1.75ex]{<}{{\setlength{\arraycolsep}{0pt}\begin{bmatrix} I&\tfrac12\,T^*&\cdots& \tfrac12\,T^{*(n-1)}&\tfrac12\,T^{*n} \\
\tfrac12\,T& I & & &\vdots \\
\vdots & & \ddots & &\vdots \\
\vdots & &  & I & \tfrac12\,T^*\\
\tfrac12\,T^n&\tfrac12\,T^{n-1}&\cdots &\tfrac12\,T & I\end{bmatrix}\begin{bmatrix}\xi\\ \xi_1\\ \vdots\\ \vdots \\ \xi_n\end{bmatrix},
\begin{bmatrix}\xi\\ \xi_1\\ \vdots\\ \vdots \\ \xi_n\end{bmatrix}}
}{>} \right\}.
\end{align*}
We can write the matrix $A$  as
\[
\setlength{\arraycolsep}{1pt}A=\begin{bmatrix} I\ & T^*&\cdots& \cdots& T^{*n} \\
0& I & & &\vdots \\
\vdots & & \ddots & &\vdots \\
\vdots & &  & I &  T^*\\
0&0&\cdots &0 & I\end{bmatrix}
\setlength{\arraycolsep}{0pt}\begin{bmatrix} I&-\tfrac12\,T^*&0&\cdots& 0 \\
-\tfrac12\,T& I & & \ddots&\vdots \\
0 & & \ddots & &0 \\
\vdots & \ddots&  & I & -\tfrac12\,T^*\\
0&\cdots& 0&-\tfrac12\,T & I\end{bmatrix}
\setlength{\arraycolsep}{2pt}\begin{bmatrix}
I&0&\cdots&\cdots&0\\
T&I& 0&\cdots&0\\
\vdots& & \ddots&\ddots \\
\vdots & & & I &0\\
T^n&\cdots&\cdots&T&I
\end{bmatrix}.
\]
As the invertible triangular block matrix preserves the first entry of the $n$-tuple vector,
and as the infimum   is taken over \emph{all}   $n$-tuples in $H$, we obtain
\begin{equation}\label{equation: ando lemma 1:1}
\begin{aligned}
\langle X_n\xi,\xi\rangle
&=\inf\left\{  \scaleleftright[1.75ex]{<}{{\setlength{\arraycolsep}{2pt}\setlength{\arraycolsep}{0pt}\begin{bmatrix} I&\tfrac12\,T^*&0&\cdots& 0 \\
\tfrac12\,T& I & & \ddots&\vdots \\
0 & & \ddots & &0 \\
\vdots & \ddots&  & I & \tfrac12\,T^*\\
0&\cdots& 0&\tfrac12\,T & I\end{bmatrix}\begin{bmatrix}\xi\\ \xi_1\\ \vdots\\ \vdots \\ \xi_n\end{bmatrix},
\begin{bmatrix}\xi\\ \xi_1\\ \vdots\\ \vdots \\ \xi_n\end{bmatrix}}
}{>} :\  \xi_1,\ldots,\xi_n\in H\right\} \abajo
&=\inf\{\langle R_n\tilde\xi,\tilde\xi\rangle:\ \tilde\xi=\begin{bmatrix}\xi&\xi_1&\cdots&\xi_n\end{bmatrix}^\transpose ,\ \xi_1,\ldots,\xi_n\in H\},
\end{aligned}
\end{equation}
where we use $R_n$ to denote the $(n+1)\times(n+1)$ tri-diagonal block-matrix from the previous line.
We were able to remove the negative signs because the infimum is taken over all $n$-tuples $\xi_1,\ldots,\xi_n\in H$; in particular, we can use $\xi_1,-\xi_2,\xi_3,-\xi_4,\ldots,(-1)^{n+1}\xi_n$.

Now we will use to our advantage the fact that each $R_n$ contains $R_{n-1}$ in its lower right corner. We have
\begin{equation}\label{equation: ando lemma 1:2}
\begin{aligned}
\langle R_n\tilde\xi,\tilde\xi\rangle
&=\langle\xi,\xi\rangle+\re\langle T\xi,\xi_1\rangle + \langle R_{n-1}\hat\xi,\hat\xi\rangle,
\end{aligned}
\end{equation}
with $\hat\xi=\begin{bmatrix}\xi_1&\cdots&\xi_n\end{bmatrix}^\transpose $. Fix $\varepsilon>0$; by \eqref{equation: ando lemma 1:1}, applied to $n-1$ and $\xi_1$, we can choose $\xi_2,\ldots,\xi_n\in H$ so that $\langle R_{n-1}\hat\xi,\hat\xi\rangle\leq\langle X_{n-1}\xi_1,\xi_1\rangle+\varepsilon$. Thus,
\begin{align*}
\left\langle\begin{bmatrix}I& \tfrac12\,T^* \\ \tfrac12\,T & X_{n-1}\end{bmatrix}\begin{bmatrix}\xi\\ \xi_1\end{bmatrix},
\begin{bmatrix}\xi\\ \xi_1\end{bmatrix}\right\rangle
&=\langle\xi,\xi\rangle+\re\langle T\xi,\xi_1\rangle
+\langle X_{n-1}\xi_1,\xi_1\rangle \abajo
&\geq \langle\xi,\xi\rangle+\re\langle T\xi,\xi_1\rangle + \langle R_{n-1}\hat\xi,\hat\xi\rangle -\varepsilon \abajo
&=\langle R_n\tilde\xi,\tilde\xi\rangle - \varepsilon
\geq \langle X_n\xi,\xi\rangle-\varepsilon.
\end{align*}
For an arbitrary $\xi_1,\ldots,\xi_n$,
\begin{align*}
\left\langle\begin{bmatrix}I& \tfrac12\,T^* \\ \tfrac12\,T & X_{n-1}\end{bmatrix}\begin{bmatrix}\xi\\ \xi_1\end{bmatrix},
\begin{bmatrix}\xi\\ \xi_1\end{bmatrix}\right\rangle
&=\langle\xi,\xi\rangle+\re\langle T\xi,\xi_1\rangle
+\langle X_{n-1}\xi_1,\xi_1\rangle \abajo
&\leq \langle\xi,\xi\rangle+\re\langle T\xi,\xi_1\rangle + \langle R_{n-1}\hat\xi,\hat\xi\rangle \abajo
&=\langle R_n\tilde\xi,\tilde\xi\rangle
\end{align*}

It follows from \eqref{equation: ando lemma 1:1} and the last two estimates (writing $\eta$ instead of $\xi_1$), that
\begin{equation}\label{equation: ando lemma 1:3}
\langle X_n\xi,\xi\rangle = \inf\left\{\left\langle\begin{bmatrix}I& \tfrac12\,T^* \\ \tfrac12\,T & X_{n-1}\end{bmatrix}\begin{bmatrix}\xi\\ \eta\end{bmatrix},
\begin{bmatrix}\xi\\ \eta\end{bmatrix}\right\rangle :\ \eta\in H\right\}.
\end{equation}
Taking limit in \eqref{equation: ando lemma 1:3},
\[
\langle X\xi,\xi\rangle = \inf\left\{\left\langle\begin{bmatrix}I& \tfrac12\,T^* \\ \tfrac12\,T & X\end{bmatrix}\begin{bmatrix}\xi\\ \eta\end{bmatrix},
\begin{bmatrix}\xi\\ \eta\end{bmatrix}\right\rangle :\ \eta\in H\right\},
\]
as claimed in \eqref{equation: ando lemma 1:0}.

We have, for all $\xi,\eta\in H$,
\begin{equation}\label{equation: ando lemma 1:4}
\left\langle\begin{bmatrix}X&0\\0&0\end{bmatrix}\begin{bmatrix}\xi\\ \eta\end{bmatrix},\begin{bmatrix}\xi\\ \eta\end{bmatrix}\right\rangle=\langle X\xi,\xi\rangle\leq\left\langle\begin{bmatrix}I& \tfrac12\,T^* \\ \tfrac12\,T & X\end{bmatrix}\begin{bmatrix}\xi\\ \eta\end{bmatrix},
\begin{bmatrix}\xi\\ \eta\end{bmatrix}\right\rangle.
\end{equation}
Thus
\[
\begin{bmatrix}I-X&\tfrac12\,T^* \\ \tfrac12\,T &X\end{bmatrix}\geq0,
\]
and $X$ belongs to the set in \eqref{equation: ando lemma 1:00}.
Now assume that $0\leq Y\leq I$ and that $\begin{bmatrix}I-Y&\tfrac12\,T^* \\ \tfrac12\,T &Y\end{bmatrix}\geq0$; this we can write as $\begin{bmatrix}I&\tfrac12\,T^* \\ \tfrac12\,T &Y\end{bmatrix}\geq\begin{bmatrix} Y&0\\0&0\end{bmatrix}$. By assumption, $X_0=I\geq Y$. Suppose that $X_{n-1}\geq Y$. Then, for each $\xi\in H$,
\begin{equation*}
\begin{aligned}
\langle X_n\xi,\xi\rangle
&=\inf\left\{\left\langle\begin{bmatrix}I& \tfrac12\,T^* \\ \tfrac12\,T & X_{n-1}\end{bmatrix}\begin{bmatrix}\xi\\ \eta\end{bmatrix},
\begin{bmatrix}\xi\\ \eta\end{bmatrix}\right\rangle :\ \eta\in H\right\} \abajo
&=\inf\{ \langle\xi,\xi\rangle+\re\langle T\xi,\eta\rangle+\langle X_{n-1}\eta,\eta\rangle: \eta\in H\} \abajo
&\geq\inf\{ \langle\xi,\xi\rangle+\re\langle T\xi,\eta\rangle+\langle Y\eta,\eta\rangle: \eta\in H\} \abajo
&=\inf\left\{\left\langle\begin{bmatrix}I& \tfrac12\,T^* \\ \tfrac12\,T &
Y\end{bmatrix}\begin{bmatrix}\xi\\ \eta\end{bmatrix},
\begin{bmatrix}\xi\\ \eta\end{bmatrix}\right\rangle :\ \eta\in H\right\} \abajo
&\geq\inf\left\{\left\langle\begin{bmatrix}Y& 0\\ 0&
0\end{bmatrix}\begin{bmatrix}\xi\\ \eta\end{bmatrix},
\begin{bmatrix}\xi\\ \eta\end{bmatrix}\right\rangle :\ \eta\in H\right\} \abajo
&=\langle Y\xi,\xi\rangle.
\end{aligned}
\end{equation*}
It follows by induction that $X_n\geq Y$ for all $n$, and thus $X\geq Y$.
\end{proof}

If the manipulations in the proof of  \cref{lemma: ando lemma 1} were not impressive enough, Ando keeps going at it, with the following striking characterization of the numerical radius:

\begin{theorem}[Ando \cite{ando1973}]\label{theorem: ando's characterization}
Let $T\in B(H)$. Then the following statements are equivalent:
\begin{enumerate}
\item\label{theorem: ando's characterization:1} $w(T)\leq1$;
\item\label{theorem: ando's characterization:2} there exist $Y,Z\in B(H)$ with $Y$ selfadjoint, $\|Y\|\leq1$, and $\|Z\|\leq1$ such that  \[T=(I+Y)^{1/2}Z(I-Y)^{1/2}.\]
\end{enumerate}
When the above conditions are satisfied, the set
\begin{equation}\label{equation: ando's characterization:4}
\{Y=Y^*:\ \exists Z,\ \|Z\|\leq1\ \text{ and }T=(I+Y)^{1/2}Z(I-Y)^{1/2}\}
\end{equation}
 admits a maximum $Y_{\max}$ and a minimum $Y_{\min}$. The corresponding $Z_{\max}$ is isometric on the range of $I-Y_{\max}$, and $Z_{\min}$ is isometric on the range of $I+Y_{\min}$.
\end{theorem}
\begin{proof}
\eqref{theorem: ando's characterization:1}$\implies$ \eqref{theorem: ando's characterization:2}. By \cref{lemma: ando lemma 1} there exists a positive contraction $X$ satisfying \eqref{equation: ando lemma 1:0}.We can write this as
\[
\langle X\xi,\xi\rangle=\inf\{\langle\xi,\xi\rangle+\re\langle T\xi,\eta\rangle+\langle X\eta,\eta\rangle:\ \eta\in H\}.
\]
Flipping a few terms around, we get
\[
\langle (I-X)\xi,\xi\rangle = \sup\{-\re\langle T\xi,\eta\rangle-\|X^{1/2}\eta\|^2:\ \eta\in H\}.
\]
Since $\eta$ moves over all of $H$, and since the second term is invariant if we replace $\eta$ with $\lambda\eta$ for $\lambda\in\TT$, we get
\begin{equation}\label{equation: ando second form of the sup}
\|(I-X)^{1/2}\xi\|^2=\sup\{|\langle T\xi,\eta\rangle|-\|X^{1/2}\eta\|^2:\ \eta\in H\}.
\end{equation}
As we can write $t\eta$ instead of $\eta$, we have shown that for a fixed $\eta$
\begin{equation}\label{equation: ando quadratic}
\|(I-X)^{1/2}\xi\|^2-t|\langle T\xi,\eta\rangle|+t^2\|X^{1/2}\eta\|^2\geq0,\ \ t\in \RR.
\end{equation}
The discriminant inequality for this quadratic is   \[ |\langle T\xi,\eta\rangle|^2\leq 4\|(I-X)^{1/2}\xi\|^2\,\|X^{1/2}\eta\|^2,\] which we write as
\begin{equation}\label{equation: ando's characterization:1}
\tfrac12\,|\langle T\xi,\eta\rangle|\leq \|(I-X)^{1/2}\xi\|\,\|X^{1/2}\eta\|.
\end{equation}
Because of the supremum in \eqref{equation: ando second form of the sup}, for any given $\xi$ there exists   $\eta$ such that the quadratic in \eqref{equation: ando quadratic} is arbitrarily close to zero. Thus the discriminant can be made arbitrarily close to zero by such an $\eta$, and the inequality in \eqref{equation: ando's characterization:1} can be made arbitrarily close to an equality. So
\begin{equation}\label{equation: ando's characterization:2}
\|(I-X)^{1/2}\xi\|=\sup\left\{\frac{\tfrac12\,|\langle T\xi,\eta\rangle|}{\|X^{1/2}\eta\|}:\ X^{1/2}\eta\ne0\right\}.
\end{equation}

We now construct  a densely-defined sesquilinear form  in the following way. Let $H_0,H_1\subset H$ be the following dense (due to $X$ being selfadjoint) linear manifolds:
\begin{align*}
H_0&=\{\xi_0+(I-X)^{1/2}\xi:\ \xi_0\in \ker (I-X),\ \xi\in H\}, \abajo
H_1&=\{\eta_0+X^{1/2}\eta:\ \eta_0\in \ker X,\ \eta\in H\}.
\end{align*}
Then we define, on $H_0\times H_1$, a form
\begin{equation}\label{equation: ando's characterization:form}
[\xi_0+(I-X)^{1/2}\xi,\eta_0+X^{1/2}\eta]:=\frac12\,\langle T\xi,\eta\rangle.
\end{equation}
By \eqref{equation: ando's characterization:1} the above form is well-defined and, since the kernel and range of $X$ are orthogonal to each other,
\begin{align*}
|[\xi_0+(I-X)^{1/2}\xi,\eta_0+X^{1/2}\eta]|&=\frac12\,|\langle T\xi,\eta\rangle|\leq\|(I-X)^{1/2}\xi\|\,\|X^{1/2}\eta\| \abajo
&\leq \|\xi_0+(I-X)^{1/2}\xi\|\,\|\eta_0+X^{1/2}\eta\|.
\end{align*}
The sesquilinear form is thus bounded with norm at most one: we can then extended it   to all of $H\times H$. By the Riesz Representation Theorem there exists a linear contraction $Z\in B(H)$, with $Z|_{\ker(I-X)}=0$ and $Z^*|_{\ker X}=0$, and such that
\begin{equation}\label{equation: ando's characterization:3}
\tfrac12\,\langle T\xi,\eta\rangle = \langle Z(I-X)^{1/2}\xi,X^{1/2}\eta\rangle,\ \ \ \ \ \xi,\eta\in H.
\end{equation}
In particular, $\tfrac12\,T=X^{1/2}Z(I-X)^{1/2}$. If we now let $Y=2X-I$, then $Y=Y^*$ and
\[
(I+Y)^{1/2}Z(I-Y)^{1/2}=(2X)^{1/2}Z(2I-2X)^{1/2}=2X^{1/2}Z(I-X)^{1/2}=T.
\]
Using \eqref{equation: ando's characterization:2} and \eqref{equation: ando's characterization:3}, for a fixed $\xi\in H$ and $\varepsilon>0$ there exists $\eta\in H$ with
\begin{align*}
\|(I-X)^{1/2}\xi\|
&\leq\frac{\tfrac12\,|\langle T\xi,\eta\rangle|}{\|X^{1/2}\eta\|}+\varepsilon
=\frac{|\langle Z(I-X)^{1/2}\xi,X^{1/2}\eta\rangle|}{\|X^{1/2}\eta\|}+\varepsilon \abajo
&\leq\|Z(I-X)^{1/2}\xi\|+\varepsilon.
\end{align*}
But $Z$ is a contraction and we can do this for all $\varepsilon>0$, so $\|Z(I-X)^{1/2}\xi\|=\|(I-X)^{1/2}\xi\|$ for all $\xi\in H$; a fortiori, as we can replace $\xi$ with $(I-X)^{1/2}\xi$, and writing $Z_{\max}$ for the contraction $Z$ we constructed, we get that \[
\|Z_{\max}(I-X)\xi\|=\|(I-X)\xi\|,\ \ \  \text{ for all }\xi\in H.
\] From $I-X=\tfrac12\,(I-Y)$, we obtain that $Z_{\max}$ is isometric on the range of $I-Y$. This $Y=2X-I$ we constructed, that we will denote as $Y_{\max}$, is  the maximum of the set in \eqref{equation: ando's characterization:4}. Indeed, if $T=(I+Y_0)^{1/2}Z_0(I-Y_0)^{1/2}$ for selfadjoint $Y_0$ and contractive $Z_0$, then
\begin{align*}
\begin{bmatrix}
I-\tfrac12\,(I+Y_0)&\tfrac12\,T^*\\ \tfrac12\,T&\tfrac12 (I+Y_0)\end{bmatrix}
&=\tfrac12\,\begin{bmatrix} I-Y_0 & R \\
R^* & I+Y_0\end{bmatrix} \abajo
& = M\,\begin{bmatrix} I & Z_0\\ Z_0^* & I\end{bmatrix}\,M^*\geq0,
\end{align*}
where $R=(I-Y_0)^{1/2} Z_0^*(I+Y_0)^{1/2}$ and $M=\frac1{\sqrt2}\begin{bmatrix}(I-Y_0)^{1/2} &0\\0& (I+Y_0)^{1/2}\end{bmatrix}$. By the maximality of $X$ in \eqref{equation: ando lemma 1:00}, we have $\tfrac12\,(I+Y_0)\leq X$, which we can write as
\[
Y_0\leq 2X-I=Y_{\max}.
\]
All of the above can be done for $T^*$, so there is a maximum, say $Y_*$, corresponding to  $T^*$. By taking the adjoint, we can rewrite any decomposition $T=(I+Y)^{1/2}Z(I-Y)^{1/2}$ as $T^*=(I+(-Y))^{1/2}Z^*(I-(-Y))^{1/2}$. It follows that $-Y\leq Y_*$ for all $Y$ that give a decomposition of $T$. In other words, $-Y_*=Y_{\min}$.

\eqref{theorem: ando's characterization:2}$\implies$ \eqref{theorem: ando's characterization:1} If
$T=(I+Y)^{1/2}Z(I-Y)^{1/2}$ for a contraction $Z$ then, using the trivial number inequality $|ab|\leq\frac12\,(|a|^2+|b|^2)$,
\begin{align*}
|\langle T\xi,\xi\rangle| &= |\langle Z(I-Y)^{1/2}\xi,(I+Y)^{1/2}\xi\rangle|
\leq\|(I-Y)^{1/2}\xi\|\,\|(I+Y)^{1/2}\xi\| \abajo
&\leq\frac12\,(\|(I-Y)^{1/2}\xi\|^2+\|(I+Y)^{1/2}\xi\|^2) \abajo
&=\frac12\,(\langle (I-Y)\xi,\xi\rangle+\langle (I+Y)\xi,\xi\rangle)
=\langle \xi,\xi\rangle. \qedhere
\end{align*}
\end{proof}

\begin{remark}\label{remark: Ando's X,Y,Z for E_21}
Let us find the above decomposition for the case $T=2E_{21}$. We have $w(T)=1$, so the above results apply. In light of \cref{remark: E_21 as a corner of the bilateral shift}, we may take $K=\ell^2(\ZZ)$, $U$ the bilateral shift, and $H=\spann\{e_1,e_2\}$. Then $U^*H=\spann\{e_{0},e_1\}$, and
\[
\spann\bigcup_{k=1}^n U^{*k}H=\spann\{e_{-n+1},e_{-n+2},\ldots,e_1\}.
\]
So, in \cref{lemma: ando lemma 1}, $I-Q_n=\sum_{k=n}^\infty E_{-k,-k}+\sum_{k=2}^\infty E_{kk}$ and $P_H=E_{11}+E_{22}$. We get that
$X_n=P_H(I-Q_n)P_H=E_{22}$ for all $n$. So $X=E_{22}$. Then $\left\langle X\begin{bmatrix}\alpha\\ \beta\end{bmatrix},\begin{bmatrix}\alpha\\ \beta\end{bmatrix}\right\rangle=|\beta|^2$; and
\[
\scaleleftright[1.75ex]{<}{\begin{bmatrix} I&\tfrac12\,T^*\\ \tfrac12\,T&X\end{bmatrix}\begin{bmatrix}\alpha\\ \beta\\ \gamma\\ \delta\end{bmatrix},\begin{bmatrix}\alpha\\ \beta\\ \gamma\\ \delta\end{bmatrix} }{>}
=\scaleleftright[1.75ex]{<}{\begin{bmatrix} 1&0&0&1\\ 0&1&0&0\\ 0&0&0&0\\ 1&0&0&1\end{bmatrix}\begin{bmatrix}\alpha\\ \beta\\ \gamma\\ \delta\end{bmatrix},\begin{bmatrix}\alpha\\ \beta\\ \gamma\\ \delta\end{bmatrix} }{>}=|\alpha+\delta|^2+|\beta|^2,
\]
with the infimum over $\gamma,\delta$ being $|\beta|^2$ (achieved when $\delta=-\alpha$). If $Y$ satisfies \[\begin{bmatrix}I-Y&E_{12}\\ E_{21}&Y\end{bmatrix}\geq0,\]
we immediately get from diagonal entries that $0\leq Y\leq I$, and considering the $2\times 2$ matrix formed by the corner entries, we have
\[
\begin{bmatrix} 1-Y_{11} &1 \\ 1& Y_{22}\end{bmatrix}\geq0.
\]
This can only be satisfied if $Y_{11}=0$, $Y_{22}=1$. From $Y_{11}=0$ and positivity, we obtain $Y_{12}=Y_{21}=0$ (a zero in the main diagonal forces its column and row to be zero). Thus $Y=E_{22}=X$. We have shown that, in this example, the set in \eqref{equation: ando lemma 1:00} consists of just $X$.

Looking into \cref{theorem: ando's characterization}, we have $Y=2X-I=E_{22}-E_{11}$. The proof in that theorem constructs $Z$ via the bilinear form on $H_0\times H_1$. In this case, $(I-X)^{1/2}=E_{11}$, $X^{1/2}=E_{22}$, so the form is
\begin{align*}
\left[\begin{bmatrix}\alpha \\ \beta\end{bmatrix},\begin{bmatrix}\gamma \\ \delta\end{bmatrix} \right]
&=\tfrac12\,\left\langle 2E_{21}\begin{bmatrix}\alpha \\ \beta\end{bmatrix},\begin{bmatrix}\gamma \\ \delta\end{bmatrix}\right\rangle
=\left\langle E_{21}\begin{bmatrix}\alpha \\ \beta\end{bmatrix},\begin{bmatrix}\gamma \\ \delta\end{bmatrix}\right\rangle \abajo
&=\left\langle E_{21}\,E_{11}\begin{bmatrix}\alpha \\ \beta\end{bmatrix},E_{22}\,\begin{bmatrix}\gamma \\ \delta\end{bmatrix}\right\rangle.
\end{align*}
Thus $Z=E_{21}$. The maximal decomposition is then
\[
2E_{21}=(I+Y)^{1/2}Z(I-Y)^{1/2}=(2E_{22})^{1/2}E_{21}(2E_{11})^{1/2}
\]
and, as we said,
\[
Y_{\max}=2X-I=E_{22}-E_{11}=\begin{bmatrix}-1&0\\0&1\end{bmatrix}.
\]
The computation that showed that $2X-I$ is maximum in the proof of \cref{theorem: ando's characterization} implies that if $X$ is unique, so is $Y$. Thus $Y_{\min}=Y_{\max}$ in this case.
\end{remark}

\begin{remark}
If instead we consider $T=E_{21}$, the situation is very different. It seems hard to follow the path all the way from \cref{theorem: positive-definite functions are compressions of unitary representations} to \cref{lemma: ando lemma 1} to find $X$ explicitly. Still, from \cref{lemma: ando lemma 1}, we know that we need to find the maximum of those selfadjoint $X$ such that $0\leq X\leq I$ and
\begin{equation}\label{equation: remark: Ando's X,Y,Z for E_21}
\begin{bmatrix} I-X & \tfrac12\,E_{12} \\ \tfrac12\,E_{21} & X\end{bmatrix}
=\begin{bmatrix} 1-X_{11}& -X_{12} & 0 &1/2\\ -X_{21} & 1-X_{22} & 0 &0 \\ 0 & 0 & X_{11} & X_{12} \\
1/2 & 0 & X_{21} & X_{22} \end{bmatrix}\geq0.
\end{equation}

It is easy to check that $X_0=\tfrac34\,E_{11}+E_{22}$ satisfies \eqref{equation: remark: Ando's X,Y,Z for E_21}, and so $X\geq\tfrac34\,E_{11}+E_{22}$.  We get immediately that $X_{11}\geq3/4$ and $X_{22}=1$ (since $I-X\geq0$). Again from $I-X\geq0$,
\[
\begin{bmatrix} 1-X_{11} & X_{12} \\ \overline{X_{12}} & 0\end{bmatrix}\geq0,
\]
and so $X_{12}=0$ by the positivity. If we now look, in \eqref{equation: remark: Ando's X,Y,Z for E_21}, at the $2\times 2$ matrix formed by the corner entries, we have
\[
\begin{bmatrix} 1-X_{11}& 1/2\\ 1/2 & 1\end{bmatrix}\geq0.
\]
Then $1-X_{11}\geq 1/4$, i.e., $X_{11}\leq 3/4$. So $X_{11}=3/4$, and $X=\begin{bmatrix} 3/4&0\\0&1\end{bmatrix}$. Now
\[
Y_{\max}=2X-I=\begin{bmatrix} 1/2&0\\0&1\end{bmatrix}.
\]
If we look again at the proof of \cref{theorem: ando's characterization}, we have $(I-X)^{1/2}=\tfrac12\,E_{11}$, $X^{1/2}=\tfrac{\sqrt3}2\,E_{11}+E_{22}$. Writing \eqref{equation: ando's characterization:3} explicitly, we immediately get $Z=E_{21}$. The maximal decomposition is then
\[
E_{21}=(I+Y_{\max})^{1/2} Z (I-Y_{\max})^{1/2}=\begin{bmatrix}\sqrt{3/2}&0\\0&\sqrt2\end{bmatrix}
\begin{bmatrix}0&0\\1&0\end{bmatrix}\begin{bmatrix}1/\sqrt2&0\\0&0\end{bmatrix}.
\]
As opposed to the previous case, though, different decompositions are possible. If we repeat the analysis above for $T^*=E_{12}$, we find that $X_*$ is now $E_{11}+\tfrac34\,E_{22}$. Its maximum $Y_*$ will be $2X_*-I=E_{11}+\tfrac12\,E_{22}$. Then, with respect to our original $T=E_{21}$, we have
\[
Y_{\min}=-Y_*=\begin{bmatrix}-1&0\\0&-1/2\end{bmatrix}.
\]
We can also get decompositions that do not come from $Y_{\max}$ nor $Y_{\min}$. For a trivial one, take $Y_0=0$, $Z_0=E_{21}$. Then, of course, $E_{21}=(I+0)^{1/2}E_{21}(I-0)^{1/2}$. Yet another fairly trivial decomposition can be found if $Y=E_{22}$. Then $(I+Y)^{1/2}=E_{11}+\sqrt2\,E_{22}$, and $(I-Y)^{1/2}=E_{11}$. So we get the decomposition
\[
E_{21}=\begin{bmatrix} 1&0\\0&\sqrt2\end{bmatrix}\begin{bmatrix}0&0\\1/\sqrt2&0\end{bmatrix}
\begin{bmatrix}1&0\\0&0\end{bmatrix}.
\]

\end{remark}

\bigskip

The following result is a well-known matricial characterization of the numerical radius. It uses \cref{theorem: ando's characterization} in an essential way.

\begin{corollary}\label{theorem: ando}
Let $T\in B(H)$. Then the following statements are equivalent:
\begin{enumerate}
\item $w(T)\leq1/2$;
\item there exists $A\in B(H)^+$, with $A\leq I$, such that
$
\begin{bmatrix}A&T^*\\ T& I-A\end{bmatrix}\geq0.
$
\end{enumerate}
\end{corollary}
\begin{proof}
If $\begin{bmatrix}A&T^*\\ T& I-A\end{bmatrix}\geq0$, then for any $\xi,\eta\in H$
\[
0\leq \left\langle \begin{bmatrix}A&T^*\\ T& I-A\end{bmatrix}\begin{bmatrix} \xi\\ \eta\end{bmatrix},\begin{bmatrix}\xi\\ \eta\end{bmatrix}\right\rangle
=\langle A\xi,\xi\rangle+\langle (I-A)\eta,\eta\rangle+2\re\langle T\xi,\eta\rangle.
\]
Taking $\eta=\lambda\xi$ for  $\lambda\in\TT$  such that $\langle T\xi,\xi\rangle=\lambda\,|\langle T\xi,\xi\rangle|$, we get
\[
0\leq\|\xi\|^2-2|\langle T\xi,\xi\rangle|,
\]
implying $w(T)\leq1/2$.

Conversely, if $w(T)\leq1/2$, by \cref{theorem: ando's characterization} there exist a selfadjoint contraction $Y$ and  contraction $Z$ with $2T=(I+Y)^{1/2}Z(I-Y)^{1/2}$. Since $Z$ is contractive, the matrix $\begin{bmatrix} I&Z^*\\ Z&I\end{bmatrix}$ is positive; then
\begingroup
\setlength\arraycolsep{2pt}
\begin{align*}
\begin{bmatrix}I+Y&2T\\2T^*&I-Y\end{bmatrix}
&=\begin{bmatrix}I+Y&(I+Y)^{1/2}Z^*(I-Y)^{1/2}\\(I-Y)^{1/2}Z(I+Y)^{1/2}&I-Y\end{bmatrix}\\
&=\begin{bmatrix}(I+Y)^{1/2}&0\\0&(I-Y)^{1/2}\end{bmatrix}
\begin{bmatrix}I&Z^*\\ Z&I\end{bmatrix}
\begin{bmatrix}(I+Y)^{1/2}&0\\0&(I-Y)^{1/2}\end{bmatrix}\\
&\geq0.
\end{align*}
\endgroup
Multiplying by $1/2$ we get
$\begin{bmatrix}A&T^*\\ T& I-A\end{bmatrix}\geq0$, where $A=\frac12\,(I+Y)$.
\end{proof}

Our main use of this result is \cref{corollary: matricial range of the 2x2 shift}, characterizing the matricial range of $E_{21}$. The use of \cref{theorem: ando} to characterize $\WW(E_{21})$ is a very well-known result, though we are not aware of any published reference.

\begin{corollary}\label{corollary: matricial range of the 2x2 shift}
For any $T\in B(H)$, the following statements are equivalent:
\begin{enumerate}
\item\label{corollary: matricial range of the 2x2 shift:1} $w(T)\leq1/2$;
\item\label{corollary: matricial range of the 2x2 shift:2} there exists $\varphi:M_2(\mathbb C)\to B(H)$, ucp, with $T=\varphi(E_{21})$.
\end{enumerate}
\end{corollary}
\begin{proof}
\eqref{corollary: matricial range of the 2x2 shift:1}$\implies$ \eqref{corollary: matricial range of the 2x2 shift:2}. \ By \cref{theorem: ando}, there exists $A\in B(H)^+$ with $A\leq I$ and $\begin{bmatrix}A&T^*\\T&I-A\end{bmatrix}\geq0$. Now define a linear map $\varphi:M_2(\mathbb C)\to B(H)$ by
\[
\varphi(E_{11})=A,\ \varphi(E_{21})=T,\ \varphi(E_{12})=T^*,\ \varphi(E_{22})=I-A.
\]
Then $\varphi$ is unital, and by Choi's criterion (\cref{proposition: choi characteriztation for the paper}) it is completely positive, since
\[
\varphi^{(2)}\left(\begin{bmatrix}E_{11}&E_{12}\\E_{21}&E_{22}\end{bmatrix}\right)=\begin{bmatrix}
A&T^*\\T&I-A\end{bmatrix}\geq0.
\]

\bigskip

\eqref{corollary: matricial range of the 2x2 shift:2}$\implies$ \eqref{corollary: matricial range of the 2x2 shift:1}. \ Let $A=\varphi(E_{11})$. As $\varphi$ is ucp, Choi's criterion  (\cref{proposition: choi characteriztation for the paper}) implies that
\[
\begin{bmatrix}
A&T^*\\T&I-A\end{bmatrix}=\varphi^{(2)}\left(\begin{bmatrix}E_{11}&E_{12}\\E_{21}&E_{22}\end{bmatrix}\right)\geq0,
\]
and so by \cref{theorem: ando}, $w(T)\leq1/2$.
\end{proof}

We remark that the proof \eqref{corollary: matricial range of the 2x2 shift:2}$\implies$ \eqref{corollary: matricial range of the 2x2 shift:1} in \cref{corollary: matricial range of the 2x2 shift} can be achieved  without appealing to Ando's nor Choi's results. Indeed, by using the Stinespring dilation one can show that if $\phi$ is ucp, then $\WW_1(\phi(T))\subset\WW_1(T)$.

\bigskip

We finish this section with another characterization due to Ando. The interesting information we find in its proof, is that the operator $C$ below allows us to express a unitary 2-dilation of $T$ explicitly.

\begin{theorem}[Ando \cite{ando1973}]\label{theorem: ando theorem 2}
Let $T\in B(H)$. Then the following statements are equivalent:
\begin{enumerate}
\item\label{theorem: ando theorem 2:1} $w(T)\leq1$;
\item\label{theorem: ando theorem 2:2} there exists $C\in B(H)$, contractive, such that $T=2(I-C^*C)^{1/2}C$.
\end{enumerate}
\end{theorem}
\begin{proof}
\eqref{theorem: ando theorem 2:1}$\implies$\eqref{theorem: ando theorem 2:2}
By \cref{theorem: ando's characterization}, there exist $Y,Z\in B(H)$, both contractive and $Y$ selfadjoint, with $Z$ isometric on the range of $(I-Y)^{1/2}$, and satisfying $T=(I+Y)^{1/2}Z(I-Y)^{1/2}$. Let $C=\tfrac1{\sqrt2}\,Z(I-Y)^{1/2}$. The fact that $Z$ is isometric on the range of $(I-Y)^{1/2}$ can be written as $(I-Y)^{1/2}Z^*Z(I-Y)^{1/2}=(I-Y)^{1/2} (I-Y)^{1/2}=I-Y$. Then
\[
I-C^*C=I-\tfrac12\,(I-Y)^{1/2}Z^*Z(I-Y)^{1/2}
=I-\tfrac12\,(I-Y)=\tfrac12\,(I+Y).
\]
Thus
\[
2(I-C^*C)^{1/2}C=(I+Y)^{1/2}Z(I-Y)^{1/2}=T.
\]

\eqref{theorem: ando theorem 2:2}$\implies$\eqref{theorem: ando theorem 2:1}
We will explicitly construct a unitary 2-dilation of $T$, and then the result will follow from \cref{theorem: sz.nagy-foias unitary 2-dilation}. We first construct a unitary dilation $W$ of $C$ on $K=\bigoplus_{k\in\ZZ} H$ as follows:
\begin{align*}
W=\sum_{\substack{k\geq1 \\ k\leq -2}} &I\otimes E_{k+1,k}
+C\otimes E_{0,0}+(I-CC^*)^{1/2}\otimes E_{0,-1} \abajo
&+ (I-C^*C)^{1/2}\otimes E_{1,0}-C^*\otimes E_{1,-1}.
\end{align*}
We encourage the reader to check that this is indeed a unitary; besides a decent amount of patience and care, the only non-trivial (but well-known) manipulation required is to note that $C(I-C^*C)^{1/2}=(I-CC^*)^{1/2}C$.

With $S$ the bilateral shift $S=\sum_{k\in\ZZ} I\otimes E_{k+1,k}$, we define $U=S^*W^2$. This is again a unitary and it has the form
\begin{align*}
U=\sum_{\substack{k\geq2 \\ k\leq -2}} &I\otimes E_{k,k-1}
+(I-C^*C)^{1/2}\otimes E_{1,0}-C^*\otimes E_{1,-1}+C^2\otimes E_{-1,0} \abajo
&+C(I-CC^*)^{1/2}\otimes E_{-1,-1} \abajo
&+(I-CC^*)^{1/2}\otimes E_{-1,-2}+(I-C^*C)^{1/2}C\otimes E_{0,0} \abajo
&+(I-C^*C)^{1/2}(I-CC^*)^{1/2}\otimes E_{0,-1}-C^*\otimes E_{0,-2}.
\end{align*}
We claim that this $U$ is a 2-dilation of $T$; that is, that $(U^n)_{0,0}=\tfrac12\,T^n$ for all $n\in\NN$. We proceed by induction. Assume that, for a fixed $n$ and for all $\ell\geq1$,
\[
(U^n)_{0,0}=\tfrac12 T^n,\ \ \ \ \ (U^n)_{0,-1}=T^{n-1}(I-C^*C)^{1/2}(I-CC^*)^{1/2},
\ \ \ \ \ (U^n)_{0,\ell}=0.
\]
This clearly holds for $n=1$, and so now we show that the above equalities for $n$ imply the corresponding versions for $n+1$. We will use repeatedly the equality $C(I-C^*C)^{1/2}=(I-CC^*)^{1/2}C$. Then (recall that our hypothesis is that $T=2(I-C^*C)^{1/2}C$ and that $U_{\ell,0}\ne0$  and $U_{\ell,-1}\ne0$ only when $\ell\in\{-1,0,1\}$)
\begin{align*}
(U^{n+1})_{0,-1}&=(U^n)_{0,-1}U_{-1,-1}+(U^n)_{0,0}U_{0,-1}+(U^n)_{0,1}U_{1,-1} \abajo
&=T^{n-1}(I-C^*C)^{1/2}(I-CC^*)^{1/2}C(I-CC^*)^{1/2} \abajo
&\ \ \ \ +\tfrac12\,T^n(I-C^*C)^{1/2}(I-CC^*)^{1/2} \abajo
&=T^{n-1}(I-C^*C)^{1/2}C(I-C^*C)^{1/2}(I-CC^*)^{1/2} \abajo
&\ \ \ \ +\tfrac12\,T^n(I-C^*C)^{1/2}(I-CC^*)^{1/2} \abajo
&=\tfrac12\,T^{n}(I-C^*C)^{1/2}(I-CC^*)^{1/2} \abajo
&\ \ \ \ +\tfrac12\,T^n(I-C^*C)^{1/2}(I-CC^*)^{1/2} \abajo
&=T^{n}(I-C^*C)^{1/2}(I-CC^*)^{1/2}.
\end{align*}
Also,
\begin{align*}
(U^{n+1})_{0,0}&=(U^n)_{0,-1}U_{-1,0}+(U^n)_{0,0}U_{0,0} + (U^n)_{0,1} U_{1,0}  \abajo
&= T^{n-1}(I-C^*C)^{1/2}(I-CC^*)^{1/2}C^2 + \tfrac12\,T^n\,\tfrac12\,T \abajo
&=T^{n-1}(I-C^*C)^{1/2}C(I-C^*C)^{1/2}C + \tfrac14\,T^{n+1} \abajo
&=\tfrac14\,T^{n+1}+\tfrac14\,T^{n+1}=\tfrac12\,T^{n+1}.
\end{align*}
And, since for $\ell\geq1$ the only $k$ such that $U_{k,\ell}\ne0$ is $k=\ell+1$ ($U_{\ell+1,\ell}=I$),
\[
(U^{n+1})_{0,\ell}=(U^n)_{0,\ell+1} U_{\ell+1,\ell}=(U^n)_{0,\ell+1}=0.
\]
The induction is then complete: for all $n\in\NN$, we have $(U^n)_{0,0}=\tfrac12\,T^n$.
\end{proof}
\section{Toeplitz Matrices}

The goal in this section is \cref{theorem: block matrix positive toeplitz}. As this is a matricial generalization of the classical \cref{theorem: positive toeplitz}, we present first the scalar version to fix ideas.

\subsection{Scalar Matrices}

The following is \cite[Lemma 2.5]{Paulsen-book}.

\begin{lemma}[Fejer-Riesz]\label{lemma: Fejer-Riesz}
Let $\tau$ be a trigonometric polynomial of the form $\tau(\lambda)=\sum_{-N}^Na_n\lambda^n$. If $\tau(\lambda)>0$ for all $\lambda\in\TT$, then there exists a polynomial $p(z)=\sum_{n=0}^Np_nz^n$ such that
\[
\tau(\lambda)=|p(\lambda)|^2,\ \ \ \lambda\in\TT.
\]
\end{lemma}
\begin{proof}
Since $\tau(\lambda)>0$, we have
\begin{align*}
\sum_{n=-N}^Na_n\lambda^n&=\re\left(\sum_{n=-N}^Na_n\lambda^n\right) \abajo
&=\sum_{n=-N}^N\re(a_n\lambda^n)
=\re(a_0)+\sum_{n=1}^N\re(a_n\lambda^n+a_{-n}\lambda^{-n}) \abajo
&=\re(a_0)+\sum_{n=1}^N\frac{(a_n+\overline{a_{-n}})\lambda^{n}+(\overline{a_n}+a_{-n})\lambda^{-n}}2 \abajo
&=\re(a_0)+\sum_{n=-N}^N\frac{a_n+\overline{a_{-n}}}2\,\lambda^n.
\end{align*}
It follows that, for all $n$, $a_n=\frac{a_n+\overline{a_{-n}}}2$, and then $a_n=\overline{a_{-n}}$. Also, $a_0\in\RR$. If necessary, we may decrease $N$ so that $a_{-N}\ne0$. Although we consider $\tau$ as a polynomial on $\TT$, its formula works of course for all $z\in \CC$. Define $g(z)=z^N\tau(z)$. Note that $g(0)=a_{-N}\ne0$, so all roots of $g$ are nonzero. Also, for $\lambda\in\TT$, $g(\lambda)=\lambda^N\tau(\lambda)\ne0$, so no zero of $g$ is in $\TT$. We have
\begin{align*}
g(1/\overline{z})&=\frac1{\overline{z}^N}\tau(1/\overline{z})
=\overline{z^{-N}\sum_{n=-N}^N\overline{a_n}z^{-n}}
=\overline{z^{-N}\sum_{n=-N}^N{a_{-n}}z^{-n}} \abajo
&=\overline{z^{-N}\sum_{n=-N}^N{a_{n}}z^{n}}
=\overline{z^{-2N}g(z)}.
\end{align*}
This implies that $g(z)=0$ if and only if $g(1/\overline{z})=0$. Since no zero is in $\TT$, we get that the zeroes of $g$ are of the form $z_1,\ldots,z_N,1/\overline{z_1},\ldots,1/\overline{z_N}$. Then
\[
g(z)=a_Nr(z)s(z),
\]
where $r,s$ are the polynomials
\[
r(z)=\prod_{n=1}^N(z-z_n),\ \ \ \ \ \ s(z)=\prod_{n=1}^N(z-1/\overline{z_n}).
\]
These two polynomials are related by
\[
\overline{s(z)}=\frac{(-1)^N\overline{z}^N r(1/\overline{z})}{z_1\cdots z_N}.
\]
Then, for $\lambda\in\TT$,
\begin{align*}
\tau(\lambda)
&=|\tau(\lambda)|=|\lambda^{-N}g(\lambda)|=|g(\lambda)|=|a_N|\,|r(\lambda)|\,|s(\lambda)| \abajo
&=|a_N|\,|r(\lambda)|\,\left|\frac{(-1)^N\lambda^{-N}r(\lambda)}{z_1\cdots z_N}\right|
=\left|\frac{a_N}{z_1\cdots z_N}\right|\,|r(\lambda)|^2.
\end{align*}
Thus $\tau(\lambda)=|p(\lambda)|^2$, where $p(z)=\left|\frac{a_N}{z_1\cdots z_N}\right|^{1/2}\,r(z)$.
\end{proof}

Matricial versions of the Fejer-Riesz Lemma exist---see for instance \cite{EphremidzeJanashiaLagvilava2009,HelsonLowdenslager1958,WienerAkutowicz1959}---but we will not discuss them here. Recall from page \pageref{definition: shift} that we denote by $S_n$ the $n\times n$ unilateral shift.

\begin{definition}\label{definition: toeplitz matrix}
An $n\times n$ {Toeplitz matrix}  is a matrix $T$ of the form
\[
T=a_0\,I_n+\sum_{k=1}^{n-1}a_k S_n^k+\sum_{k=1}^{n-1}a_{-k}{S_n^*}^k,
\]
where $a_k\in\CC$ for all $k$.
Graphically, this is
\[
T=\begin{bmatrix}
a_0 & a_{-1} & a_{-2} & \cdots & \cdots & \cdots & \cdots & a_{-n+1}\\
a_1 & a_0 & a_{-1} & a_{-2} & \ddots& \ddots& \ddots& \vdots\\
a_2 & a_1 & a_0 & a_{-1} & \ddots &\ddots &\ddots & \vdots\\
\vdots & a_2 & \ddots & \ddots & \ddots & \ddots &\ddots & \vdots\\
\vdots & \ddots& \ddots & \ddots & \ddots & \ddots & a_{-2} & \vdots\\
\vdots &\ddots &\ddots & \ddots & a_1 & a_0 & a_{-1} & a_{-2}\\
\vdots &\ddots & \ddots&\ddots & a_2 & a_1 & a_0 & a_{-1}\\
a_{n-1} & \cdots & \cdots  & \cdots & \cdots & a_2 & a_1 & a_0\\
\end{bmatrix}
\]
\end{definition}
If in particular $T$ is Hermitian, i.e. $T=T^*$, then
\begin{equation}\label{equation: hermitian toeplitz}
T=a_0\,I_n+2\re \sum_{k=1}^{n-1}a_k S_n^k, \ \ a_0,a_1,\ldots,a_{n-1}\in\CC.
\end{equation}
In the case of a Hermitian Toeplitz matrix we will write, when needed, $a_{-k}=\overline{a_k}$.

The following theorem is based on \cite[Theorem 2.14]{Paulsen-book}; we do not need to make use of this theorem, but we will use a matricial generalization, \cref{theorem: block matrix positive toeplitz} and so the scalar proof might help some readers. Paulsen considers infinite sequences, which we don't need here.

\begin{theorem}\label{theorem: positive toeplitz}
Let $T$ be a Hermitian Toeplitz matrix as in \eqref{equation: hermitian toeplitz}. Then the following statements are equivalent:
\begin{enumerate}
\item\label{theorem: positive toeplitz:1} $T$ is positive;
\item\label{theorem: positive toeplitz:2} there exists a positive linear functional $\phi$ on $C(\TT)$ such that
$a_k=\phi(z^k)$, $k=0,1,\ldots,n-1$.
\end{enumerate}
\end{theorem}
\begin{proof}
\eqref{theorem: positive toeplitz:1}$\implies$\eqref{theorem: positive toeplitz:2} Consider the operator system $\osss=\spann\{z^k:\ k\in\ZZ\}\subset C(\TT)$. Define a linear map $\phi:\osss\to\CC$ by $\phi(z^k)=a_k$,  $\phi(z^{-k})=\overline{a_k}$ for $k=0,\ldots,n-1$, and $\phi(z^k)=0$ for $|k|\geq n$. Let $\tau\in\osss$ be strictly positive, i.e. $\tau(\lambda)>0$ for all $\lambda\in\TT$. By \cref{lemma: Fejer-Riesz}, there exist $p_0,p_1,\ldots,p_m$ such that $\tau(\lambda)=\sum_{k,j=0}^mp_k\overline{p_j}\,\lambda^{k-j}$. Assume, without loss of generality, that $m\geq n$ (we complete the list of $p_k$ with zeroes if it is not the case). Then, with the convention that $a_{-k}=\overline{a_k}$, $a_k=0$ if $|k|\geq n$, and with $x=(p_0,\ldots,p_{m})^\transpose $,
\[
\phi(\tau)=\sum_{k,j=0}^{m}p_k\overline{p_j}\,a_{k-j}=\langle (T\oplus 0)x,x\rangle\geq0
\]
by the positivity of $T$. For arbitrary positive $\tau$, we have that for any $\varepsilon>0$ the function $\tau'(\lambda)=\tau(\lambda)+\varepsilon$ is strictly positive, and so $\phi(\tau)+\varepsilon=\phi(\tau')\geq0$ for all $\varepsilon>0$, which implies that $\phi(\tau)\geq0$. Thus $\phi$ is a positive linear functional on the operator system of the trigonometric polynomials; this implies that it is bounded and we can extend it by density to $C(\TT)$.

\eqref{theorem: positive toeplitz:2}$\implies$\eqref{theorem: positive toeplitz:1} Note that, since $T$ is Hermitian, $a_{-k}=\overline{a_k}=\overline{\phi(z^k)}=\phi(z^{-k})$. Given $x=(p_0,\ldots,p_{n-1})^\transpose \in\CC^n$,
\[
\langle Tx,x\rangle=\sum_{k=0}^{n-1}\sum_{j=0}^{n-1}a_{k-j}p_k\overline{p_j}
=\sum_{k=0}^{n-1}\sum_{j=0}^{n-1}\phi(z^{k-j})p_k\overline{p_j}
=\phi\left(\left|\sum_{k=0}^nz^{k}p_k\right|^2\right)\geq0.
\qedhere
\]
\end{proof}

\begin{remark}
The proof of \eqref{theorem: positive toeplitz:2}$\implies$\eqref{theorem: positive toeplitz:1} in \cref{theorem: positive toeplitz} can also be achieved by using that $\phi$ is completely positive (due to the abelian domain) and then noting that
\begin{align*}
T&=I_n+2\re \sum_{k=1}^{n-1}\phi(z^k) S_n^k\\
&=\phi^{(n)}\left(
\begin{bmatrix}
1&z&z^2&\cdots&z^{n-1}\\
0&0&&\cdots&0\\
\vdots& &&& \vdots\\
0&0&&\cdots&0\\
\end{bmatrix}^*
\begin{bmatrix}
1&z&z^2&\cdots&z^{n-1}\\
0&0&&\cdots&0\\
\vdots& &&& \vdots\\
0&0&&\cdots&0\\
\end{bmatrix}
\right)\geq0.
\end{align*}
\end{remark}

\subsection{Block  Toeplitz Matrices}

An $n\times n$ \emph{block Toeplitz matrix}  is a matrix $T$ of the form
\[
T=A_0\otimes I_n+\sum_{k=1}^{n-1}A_k\otimes S_n^k+\sum_{k=1}^{n-1}A_{-k}\otimes {S_n^*}^k,
\]
where $A_{-(n-1)},\ldots,A_0,\ldots,A_{n-1}\in \cA$ for some C$^*$-algebra $\cA$.
If in particular $T$ is Hermitian, i.e. $T=T^*$, then
\begin{equation}\label{equation: block matrix hermitian toeplitz}
T=A_0\otimes I_n+2\re \sum_{k=1}^{n-1}A_k \otimes S_n^k
\end{equation}
and we may write $A_{-k}=A_k^*$.

\cref{theorem: block matrix positive toeplitz} is the block-matrix version of \cref{theorem: positive toeplitz}. We will later use it in the proof of \cref{theorem: arveson technical result}, with $\cA=M_m(\CC)$.

\begin{theorem}\label{theorem: block matrix positive toeplitz}
Let $T$ be a Hermitian block Toeplitz matrix as in \eqref{equation: block matrix hermitian toeplitz}, with coefficients in the C$^*$-algebra $\cA$. Then the following statements are equivalent:
\begin{enumerate}
\item\label{theorem: block matrix positive toeplitz:1} $T$ is positive;
\item\label{theorem: block matrix positive toeplitz:2} there exists a (completely) positive map $\phi:C(\TT)\to \cA$ such that
\begin{equation}\label{equation: matricial phi}\phi(z^k)=A_k,\ k=0,\ldots,n-1.
\end{equation}
\end{enumerate}
\end{theorem}
\begin{proof}
\eqref{theorem: block matrix positive toeplitz:1}$\implies$\eqref{theorem: block matrix positive toeplitz:2} Consider the operator system $\osss=\spann\{z^k:\ k\in\ZZ\}\subset C(\TT)$. Define a linear map $\phi:\osss\to\cA$ by $\phi(z^k)=A_k$,  $\phi(z^{-k})={A_k^*}$ for $k=0,\ldots,n-1$, and $\phi(z^k)=0$ for $|k|\geq n$. Let $\tau\in\osss$ be strictly positive, i.e. $\tau(\lambda)>0$ for all $\lambda\in\TT$. By \cref{lemma: Fejer-Riesz}, there exist $p_0,p_1,\ldots,p_m$ such that $\tau(\lambda)=\sum_{k,j=0}^mp_k\overline{p_j}\,\lambda^{k-j}$. Assume, without loss of generality, that $m\geq n$ (we complete the list of $p_k$ with zeroes if it is not the case). Then, with the convention that $A_{-k}={A_k^*}$, $A_k=0$ if $|k|\geq n$, 
\[
\phi(\tau)=\sum_{k,j=0}^{m}p_k\overline{p_j}\,A_{k-j}
=\begin{bmatrix}p_0\,I\\ \vdots \\ p_{m}\,I\end{bmatrix}^*\,\begin{bmatrix}T\\&0_{m-n}\end{bmatrix}\,\begin{bmatrix}p_0\,I\\ \vdots \\ p_{m}\,I\end{bmatrix}\geq0
\]
by the positivity of $T$. For arbitrary positive $\tau$, we have
that for any $\varepsilon>0$ the function
$\tau'(\lambda)=\tau(\lambda)+\varepsilon$ is strictly positive,
and so $\phi(\tau)+\varepsilon=\phi(\tau')\geq0$ for all
$\varepsilon>0$, which implies that $\phi(\tau)\geq0$. Thus $\phi$
is a   positive linear map in the operator system of the trigonometric polynomials; thus it is bounded, and it   extends by density to a positive map
on $C(\TT)$ with range still contained in $\cA$.

\eqref{theorem: block matrix positive toeplitz:2}$\implies$\eqref{theorem: block matrix positive toeplitz:1} Since $C(\TT)$ is abelian, $\phi$ is completely positive. Consider a Stinespring dilation of $\phi$, i.e. $\phi(f)=V^*\pi(f)V$. Then,
for each vector $x=(\xi_0,\ldots,\xi_{n-1})^\transpose \in H^n$,
\begin{align*}
\langle Tx,x\rangle&=\sum_{k=0}^{n-1}\sum_{j=0}^{n-1}\langle A_{k-j}\xi_k,\xi_j\rangle
=\sum_{k=0}^{n-1}\sum_{j=0}^{n-1}\langle V^*\pi(z^{k-j})V \xi_k,\xi_j\rangle \abajo
&=\sum_{k=0}^{n-1}\sum_{j=0}^{n-1}\langle\pi(z^j)^*\pi(z^{k})V \xi_k,V\xi_j\rangle \abajo
&=\left\|\sum_{k=0}^{n-1} \pi(z^{k})V\xi_k\right\|^2\geq0.\qedhere
\end{align*}
\end{proof}

\section{Nilpotent Dilations and Matricial Range}\label{section: nilpotent dilations}

The following is a significant technical result by Arveson, characterizing those contractions that can be power-dilated to nilpotents. The proof does not follow the original; in particular, the argument that we offer for   \eqref{theorem: arveson technical result:3}$\implies$\eqref{theorem: arveson technical result:0} is more algebraic and direct that Arveson's original.

\begin{theorem}\cite[Theorem 1.3.1]{arveson1972}\label{theorem: arveson technical result}
Let $T\in B(H)$ with $\|T\|\leq1$ and let $n\in\NN$ with $n\geq2$. Then the following statements are equivalent:
\begin{enumerate}
\item\label{theorem: arveson technical result:0} There exists $\phi:M_n(\CC)\to B(H)$, ucp, with
$\phi(S_n^j)=T^j$, $j=1,\ldots,n-1$.
\item\label{theorem: arveson technical result:2} There exists a Hilbert space $K\supseteq H$ and $N\in B(K)$ such that $N$ is unitarily equivalent to $\bigoplus_j S_n$,  and $T^j=P_HN^j|_H$, $j=0,1,\ldots,n-1$.
\item\label{theorem: arveson technical result:1} There exists a Hilbert space $K\supseteq H$ and $N\in B(K)$ such that $\|N\|\leq1$, $N^n=0$, and $T^j=P_HN^j|_H$, $j=0,1,\ldots,n-1$.

\item\label{theorem: arveson technical result:3} $I+2\re\sum_{k=1}^{n-1}\lambda^kT^k\geq0$ for all $\lambda\in\TT$.
\end{enumerate}
\end{theorem}

\begin{proof}
\eqref{theorem: arveson technical result:0}$\implies$\eqref{theorem: arveson technical result:2}
This is a straightforward consequence of Stinespring's Dilation Theorem. Indeed, after writing $\phi(X)=P_H\pi(X) |_H$---under the usual identification of $H$ with its range under $V$---we can take $N=\pi(S_n)$ (recall that any representation of $M_n(\CC)$ is of the form $X\longmapsto X\otimes I$).

\eqref{theorem: arveson technical result:2}$\implies$\eqref{theorem: arveson technical result:1} Trivial.

\eqref{theorem: arveson technical result:1}$\implies$\eqref{theorem: arveson technical result:3}
We have that $T^j=P_H\,N^j|_H$ for $j=0,1,\ldots,n-1$. As compressions are (completely) positive, it is enough to prove the inequality $I+2\re\,\sum_{k=1}^{n-1}\lambda^kN^k\geq0$ for all $\lambda\in\TT$. Now we can take advantage of the fact that $N^n=0$. Fix $z\in\mathbb C$ with $|z|<1$. We have
\begin{align*}
I+2\re\,\sum_{k=1}^{n-1}z^kN^k&=I+2\re\,\sum_{k=1}^\infty z^kN^k=\re \left(I+ 2zN(I-zN)^{-1}\right)\\
&=\re\left([(I-zN) +2zN](I-zN)^{-1}\right)\\
&=\re (I+zN)(I-zN)^{-1}\\
&=\frac12\,\left( (I+zN)(I-zN)^{-1}+(I-\overline z N^*)^{-1}(I+\overline z N^*)\right)\\
&= 2(I-\overline z N^*)^{-1} (I-|z|^2 N^*N)(I-zN)^{-1}\geq0.\\
\end{align*}
Now given $\lambda\in \TT$, we have from above that $I+2\re\sum_{k=1}^{n-1}r^k\lambda^kT^k\geq0$ for all $r\in [0,1)$. The positivity is preserved as $r\nearrow1$.

\eqref{theorem: arveson technical result:3}$\implies$\eqref{theorem: arveson technical result:0} Define a linear map $\psi:\oss{I,S_n,\ldots,S_n^{n-1}}\to B(H)$ by
$\psi(S_n^j)=T^j$. We will show that $\psi$ is ucp. If we take $A_0,\ldots,A_{n-1}\in M_m(\CC)$ such that
$A_0\otimes I_n+2\re \sum_{k=1}^{n-1}A_k \otimes S_n^k\geq0$, then by \cref{theorem: block matrix positive toeplitz} there exists a completely positive map $\phi:C(\TT)\to M_m(\CC)$ with $\phi(z^k)=A_k$ for $k=0,\ldots,n-1$.
We have
\begin{align*}
\psi^{(m)}\left(A_0\otimes I_n+2\re \sum_{k=1}^{n-1}A_k \otimes S_n^k\right)
&=A_0\otimes \psi(I_n)+2\re \sum_{k=1}^{n-1}A_k \otimes \psi(S_n^k)\\
&=\phi(1)\otimes I+2\re \sum_{k=1}^{n-1}\phi(z^k) \otimes T^k\\
&=(\phi\otimes\id)\!\left(1\otimes I+2\re\sum_{k=1}^{n-1} z^k\otimes T^k\right).
\end{align*}
The expression inside the brackets is a function $C(\TT)\to M_n(\CC)$. For each $\lambda\in\TT$ it evaluates to
\[
1\otimes I+2\re\sum_{k=1}^{n-1} \lambda^k\otimes T^k
=1\otimes \left(I+2\re\sum_{k=1}^{n-1} \lambda^kT^k\right)\geq0.
\]
As $\phi$ is ucp,  $\phi\otimes\id$ is positive (this can be seen quickly by using Stinespring) and so the expression $(\phi\otimes\id)\,(\cdot)$ above is positive; thus $\psi^{(m)}$ is positive for any $m$, i.e., completely positive.

Now we can extend it via Arveson's Extension  Theorem to $M_n(\CC)$.
\end{proof}

In the case $n=2$, \cref{theorem: arveson technical result} characterizes the matricial range of $S_2=E_{21}=\begin{bmatrix}0&0\\1&0\end{bmatrix}$. Indeed:
\begin{corollary}\label{corollary: matricial range of E_21}
The extended matricial range of the $2\times2$ unilateral shift consists of the set of all operators $T\in B(H)$ (for any dimension of $H$) such that $w(T)\leq1/2$.
\end{corollary}
\begin{proof}
Note that $w(T)\leq1/2$ implies that $\|T\|\leq1$, since $\|T\|\leq 2w(T)$. Now combine the case $n=2$ in \cref{theorem: arveson technical result} with \cref{proposition: characterization of the numerical range}.
\end{proof}

We remark that for $n>2$, \cref{theorem: arveson technical result} does \textbf{not} characterize the matricial range of $S_n$, as one is considering the additional requirement that higher powers of $S_n$ are mapped to the corresponding powers of $T$. Forty five years after the above results, the matricial range of $S_n$, $n\geq3$, has not been characterized. As far as we can tell, there is not even a conjecture of what it  could be.

\begin{remark}
Fix $n\in\mathbb N$. The map $\gamma:\text{span}\,\{I,S_n,S_n^2,\ldots,S_n^{n-1}\}\to C(\TT )$ given by $\gamma(S_n^j)=z^j$, where $z$ is the identity function, is never completely contractive. Indeed, every power of the shift is mapped to a unitary. If $\gamma$ were ucc, we would be able to extend it to a ucp map $M_n(\mathbb C)\to X$ where $X$ is the injective envelope of $C(\TT )$. Then
\[
\gamma(E_{11})=\gamma({S_n^*}^{n-1}S_n^{n-1})\geq \gamma(S_n^{n-1})^*\gamma(S_n^{n-1})=1=\gamma(I).
\]
Positivity then implies that $\gamma(E_{22})=\cdots=\gamma(E_{nn})=0$. So, for any $k=2,\ldots,n$,
\[
0\leq\gamma(E_{k1})\gamma(E_{1k})\leq\gamma(E_{k1}E_{1k})=\gamma(E_{kk})=0,
\]
implying that $\gamma(E_{k1})=0$ for $k\geq2$. Then
\[
\gamma(S_n)=\gamma\left(\sum_{k=1}^{n-1}E_{k+1,k}\right)=0,
\]
a contradiction.
\end{remark}

\section{Characterizations of the Numerical Radius}

The following theorem requires no proof, as it just collects equivalences we proved in previous sections.
It is a consequence of very subtle ideas. 

\begin{theorem}\label{theorem: matricial range of E_12 all conditions}
Let $T\in B(H)$. The following statements are equivalent:
\begin{enumerate}
\item $w(T)\leq1$;
\item $\re(\lambda T)\leq I$ for all $\lambda\in\TT$;
\item there exists a Hilbert space $K\supset H$, and a unitary $U\in B(K)$ such that $T^n=2P_HU^n|_H$ for all $n\in\NN$;
\item there exists a Hilbert space $K\supset H$, and $N\in B(K)$, unitarily equivalent to $\bigoplus_j E_{21}$, such that $T=2P_H N|_H$;
\item there exists a Hilbert space $K\supset H$, and $N\in B(K)$ with $\|N\|\leq1$ and $N^2=0$, such that $T=2P_H N|_H$;
\item there exist contractions $Y,Z\in B(H)$, with $Y=Y^*$, such that $T=(I-Y)^{1/2}Z(I+Y)^{1/2}$;
\item there exists $A\in B(H)$, with $0\leq A\leq I$, such that $\begin{bmatrix} A&2T\\ 2T^*&I-A\end{bmatrix}\geq0$;
\item there exists $A\in B(H)$, with $0\leq A\leq I$, with $T=2(I-A^*A)^{1/2}A$;
\item there exists $\varphi:M_2(\CC)\to B(H)$, ucp, with $\varphi(E_{21})=\tfrac12\,T $.
\end{enumerate}
\end{theorem}

\begin{proof}
Combine \cref{proposition: characterization of the numerical range} with \cref{theorem: sz.nagy-foias unitary 2-dilation}, \cref{theorem: ando}, \cref{theorem: ando theorem 2},
and \cref{theorem: arveson technical result}.
\end{proof}

\section*{Acknowledgements}

This work has been supported in part by the Discovery Grant program of the Natural Sciences and Engineering Research Council of Canada. I would also like to thank D. Farenick for his encouragement and for several useful suggestions.

\bibliographystyle{abbrv} 

\end{document}